\documentclass[a4paper,12pt]{amsart}

\pdfoutput=1 

\usepackage[top=3cm,bottom=3cm,outer=3cm,inner=2cm,marginpar=2.45cm]{geometry}
\usepackage[destlabel,final,colorlinks=true]{hyperref}
\usepackage[abbrev]{amsrefs}

\usepackage[french,ngerman,english]{babel}
\usepackage[utf8]{inputenc}
\usepackage[T1]{fontenc}

\usepackage[all]{xy}
\renewcommand{\objectstyle}{\displaystyle}

\usepackage{enumitem}
\usepackage{mathtools}  
\usepackage[usenames,dvipsnames]{xcolor}
\usepackage{calc} 

\babeltags{de = ngerman}
\babeltags{fr = french}

\usepackage{upref}  
\usepackage{embrac} 

\usepackage[
]{marionotations}

\usepackage[normalem]{ulem} 

\usepackage{bbm}     

\usepackage[Smaller]{cancel} 

\usepackage{breakurl}  

\usepackage{subfiles}

\newcommand{\defaultDimension}{n}

\newcommand{\defaultAmbientSpace}{X}

\newcommand{\defaultlcIndex}{\sigma}

\newcommand{\defaultcohDegree}{q}

\newcommand{\defaultlclocus}{D}

\newcommand{\defaultvphi}{\vphi_F}
\newcommand{\setDefaultvphi}[1]{\renewcommand{\defaultvphi}{#1}}

\newcommand{\defaultpsi}{\psi_D}

\newcommand{\defaultMetric}{\omega}

\newcommand{\alert}[2][RoyalBlue]{{\color{#1}#2}}

\NewDocumentCommand{\logKX}{
  t{M} 
  o    
}{K_X \otimes D \otimes \IfNoValueTF{#2}{F \IfBooleanT{#1}{\otimes M}}{#2}}

\NewDocumentCommand{\Ltwo}{ 
  D//{\bullet,\bullet}      
  D<>{\defaultAmbientSpace} 
  s                         
  m                         
}{L^{#1}_{(2)}\paren{\IfBooleanF{#3}{#2;} #4}}

\NewDocumentCommand{\Harm}{ 
  t{'}                      
  D//{\defaultcohDegree}    
  D<>{\defaultAmbientSpace} 
  g                         
  t{,}                      
  G{\defaultvphi}           
  e{_}                      
}{\mathcal{H}^{\IfBooleanF{#1}{\defaultDimension,}#2}\IfNoValueF{#4}{\paren{#3;#4}}_{#6 \IfNoValueF{#7}{,#7}}}

\NewDocumentCommand{\lcIndex}{ 
  m  
  m  
  m  
}{#1\IfNoValueF{#2}{+#2}\IfNoValueF{#3}{-#3}}

\NewDocumentCommand{\lcData}{ 
  G{\defaultvphi}  
  O{\defaultpsi}   
  e{.}             
}{\paren{#1; \IfNoValueF{#3}{#3 \cdot} #2}}

\NewDocumentCommand{\lcdata}{ 
  s                
  d<>              
  G{\defaultvphi}  
  O{\defaultpsi}   
  e{.,}            
}{\newcommand{\datalist}{\IfNoValueF{#2}{#2,}#3,#4\IfNoValueF{#5}{,#5}\IfNoValueF{#6}{,#6}}
\IfBooleanTF{#1}{\datalist}{\paren{\datalist}}}

\newcommand{\spHbase}{\mathbb{H}}
\NewDocumentCommand{\spH}{ 
  D//{\defaultcohDegree}  
  t{M}                    
  m                       
}{\spHbase^{#1}\paren{\IfBooleanT{#2}{M\otimes}#3}}

\DeclareMathOperator{\lc}{lc} 
\NewDocumentCommand{\lcc}{ 
  D||{\defaultlcIndex}       
  e{+-}                      
  D<>{\defaultAmbientSpace}  
  t{'}                       
  D(){\defaultlclocus}       
}{\lc_{#4}^{\lcIndex{#1}{#2}{#3}}\IfBooleanTF{#5}{\paren{#6}}{\lcData}}

\NewDocumentCommand{\lcS}{  
  s                       
  D(){\defaultlclocus}    
  D||{\defaultlcIndex}    
  e{+-}                   
  d<>                     
  O{p}                    
}{\mathtt{\IfBooleanT{#1}{\rs} #2}^{\lcIndex{#3}{#4}{#5}}_{\IfNoValueF{#6}{#6,}#7}}

\NewDocumentCommand{\PRes}{ 
  O{}      
  d()      
}{\mathcal R_{#1}\IfNoValueF{#2}{\paren{#2}}}

\NewDocumentCommand{\HRes}{ 
  d()   
}{\mathfrak{R}\IfNoValueF{#1}{\paren{#1}}}

\newcommand{\defidlof}[1]{\mathcal{I}_{#1}}  
\NewDocumentCommand{\mtidlof}{   
  O{}      
  D<>{#1}  
  m        
}{\multidl_{#2}\paren{#3}} 

\NewDocumentCommand{\residlof}{  
  D||{\defaultlcIndex}   
  e{+-}                  
  d<>                    
  s                      
}{\sheaf R_{\IfNoValueTF{#4}{}{#4,} \lcIndex{#1}{#2}{#3}}\IfBooleanF{#5}{\lcData}}

\NewDocumentCommand{\Adjidlof}{
  D||{\defaultlcIndex}       
  D<>{\defaultAmbientSpace}  
  D(){\defaultlclocus}       
  m                          
}{\operatorname{\mathit{Adj}}^{#1}_{\paren{#2,#3}}\paren{#4}}

\NewDocumentCommand{\aidlof}{
  D||{\defaultlcIndex}   
  e{+-}                  
  d<>                    
  s                      
}{\sheaf{J}_{\!\IfNoValueTF{#4}{}{#4,} \lcIndex{#1}{#2}{#3}}\IfBooleanF{#5}{\lcData}}

\NewDocumentCommand{\faidlof}{
  D||{\defaultlcIndex}   
  e{+-}                  
  t{/}                   
  D||{\defaultlcIndex}   
  e{+-}                  
}{\fracAidlof{\lcIndex{#1}{#2}{#3}}{\lcIndex{#5}{#6}{#7}}}

\NewDocumentCommand{\fracAidlof}{
  m                  
  m                  
  d<>                
  s                  
  G{\defaultvphi}    
  O{\defaultpsi}     
  e{.}               
}{\frac{
    \aidlof|#1|<#3>*\IfBooleanF{#4}{\lcData{#5}[#6].{#7}}
  }{
    \aidlof|#2|<#3>*\IfBooleanF{#4}{\lcData{#5}[#6].{#7}}
  }}

\NewDocumentCommand{\lcV}{ 
  D||{\defaultlcIndex}    
  D//{\defaultvphi}       
  d()                     
  e{^}                    
  O{\defaultpsi}          
}{\:d\operatorname{lcv}^{#1\IfNoValueF{#4}{,\paren{#4}}}_{\IfNoValueF{#3}{#3,}#2}\left[#5\right]}

\NewDocumentCommand{\Ohvol}{ 
  D//{\defaultvphi} 
  d()               
  O{\defaultpsi}    
}{\dvol_{\IfNoValueF{#2}{#2,}#1}\left[#3\right]}

\newcommand{\dvol}{\:d\vol}

\NewDocumentCommand{\lcDataNormSubscript}{
  d<>                   
  s                     
  G{\defaultvphi}       
  O{\defaultpsi}        
  e{.}                  
  D||{\defaultlcIndex}  
  e{+-}                 
}{\IfNoValueF{#1}{#1,}
  \IfBooleanF{#2}{#3, \IfNoValueF{#5}{#5 \cdot} #4,}
  \lcIndex{#6}{#7}{#8}}

\newcommand{\RTFsym}{\mathfrak{F}} 
\NewDocumentCommand{\RTF}{ 
  s          
  G{\RTFsym} 
  o          
  >{\SplitArgument{1}{,}} d<> 
  d||        
  D(){\eps}  
  t{,}       
}{%
  \begingroup%
    \newif\ifsmasht%
    \IfBooleanTF{#1}{\smashttrue}{\smashtfalse}%
    \newif\ifboolup%
    \booluptrue%
    \IfNoValueT{#3}{\IfNoValueT{#4}{\IfNoValueT{#5}{\boolupfalse}}}%
    \newcommand{\supsrptstr}{\IfNoValueF{#3}{#3}\IfNoValueF{#4}{\inner#4}\IfNoValueF{#5}{\abs{#5}^2}}
    \newcommand{\RTFvar}{#6}
    #2\RTFprocess
}

\NewDocumentCommand{\RTFprocess}{
  o                     
  d<>                   
  t{,}                  
  G{\defaultvphi}       
  O{\defaultpsi}        
  e{.}                  
  D||{\defaultlcIndex}  
  e{+-}                 
}{\newcommand{\subsrptstr}{%
    \IfNoValueTF{#1}{
    \IfNoValueF{#2}{#2,}
    \IfBooleanT{#3}{#4,#5,\IfNoValueF{#6}{#6,}}
    \lcIndex{#7}{#8}{#9}}{#1}}%
  \newcommand{\srptstr}{\cramped{{}^{\supsrptstr}%
      \ifboolup _
      \fi{\ifboolup\displaystyle\fi\paren{\RTFvar}%
          \ifboolup {\scriptstyle \subsrptstr} \else _{\subsrptstr} \fi%
        }}}%
  \ifboolup%
    \ifsmasht%
      \smash[t]{
        \raisebox{\depthof{$\srptstr$} * \real{0.3}}{$\srptstr$}%
      }%
    \else%
      \raisebox{\depthof{$\srptstr$} * \real{0.3}}{$\srptstr$}%
    \fi%
  \else%
    \srptstr%
  \fi%
  \endgroup%
}

\NewDocumentCommand{\mtlog}{O{e} d() D||{\defaultpsi}}{\log\!#1^{\paren{#2}}\abs{#3}}
\NewDocumentCommand{\slog}{O{e} D||{\defaultpsi}}{\log\abs{#1 #2}}
\NewDocumentCommand{\dlog}{O{e} D||{\defaultpsi}}{\mtlog[#1](2)|#2|}

\NewDocumentCommand{\logpole}{ 
  D||{\defaultpsi}       
  D//{\defaultlcIndex}   
  E{.^}{{e}{1+\eps}}     
  s                      
}{\abs{#1}^{#2} \IfBooleanTF{#5}{\slog[#3]|#1|}{\paren{\slog[#3]|#1|}^{#4}}}

\DeclareFontFamily{OMX}{MnSymbolE}{}
\DeclareSymbolFont{MnLargeSymbols}{OMX}{MnSymbolE}{m}{n}
\SetSymbolFont{MnLargeSymbols}{bold}{OMX}{MnSymbolE}{b}{n}
\DeclareFontShape{OMX}{MnSymbolE}{m}{n}{
    <-6>  MnSymbolE5
   <6-7>  MnSymbolE6
   <7-8>  MnSymbolE7
   <8-9>  MnSymbolE8
   <9-10> MnSymbolE9
  <10-12> MnSymbolE10
  <12->   MnSymbolE12
}{}
\DeclareFontShape{OMX}{MnSymbolE}{b}{n}{
    <-6>  MnSymbolE-Bold5
   <6-7>  MnSymbolE-Bold6
   <7-8>  MnSymbolE-Bold7
   <8-9>  MnSymbolE-Bold8
   <9-10> MnSymbolE-Bold9
  <10-12> MnSymbolE-Bold10
  <12->   MnSymbolE-Bold12
}{}
\DeclareMathDelimiter{\llangle}{\mathopen}%
{MnLargeSymbols}{'164}{MnLargeSymbols}{'164}
\DeclareMathDelimiter{\rrangle}{\mathclose}%
{MnLargeSymbols}{'171}{MnLargeSymbols}{'171}

\newcommand{\iinner}[2]{\left\llangle#1,#2\right\rrangle}
\newcommand{\eqcls}[1]{\left[#1\right]}

\NewDocumentCommand{\idxup}{ 
  m                  
  O{\defaultMetric}  
  t{,}               
  o                  
  s                  
  t{.}               
}{\paren{#1}^{
    \!\IfBooleanTF{#5}{\smash[t]{#2}}{#2}\IfNoValueF{#4}{, #4}
  }\IfBooleanT{#6}{\!\!\ctrt}}

\newcommand{\dbadj}{\dbar^{\smash{\mathrlap{*}\;\:}}}

\NewDocumentCommand{\dep}{t{;} d<> O{\nu} m}{#4\IfBooleanTF{#1}{_}{^}{\IfNoValueF{#2}{#2\:}(#3)}}

\NewDocumentCommand{\sm}{s m}{{#2}\IfBooleanTF{#1}{_}{^}\text{sm}}

\NewDocumentCommand{\idx}{ 
  O{i} 
  m    
  o    
  t{.} 
  t{,} 
  o    
  m    
}{{#1}_{#2} \IfNoValueF{#3}{#3}
  \IfBooleanT{#4}{\dotsm} \IfBooleanT{#5}{,\dots,}
  \IfNoValueF{#6}{#6} {#1}_{#7}}

\newcommand{\charfct}{\mathbbm 1}

\newcommand{\cvr}[1]{\mathfrak{#1}} 
\NewDocumentCommand{\rs}{ 
  s  
  m  
}{\IfBooleanTF{#1}{\smash[t]{\widetilde{#2}}}{\widetilde{#2}}}

\DeclareMathOperator{\mlc}{mlc} 
\newcommand{\Diff}{\operatorname{Diff}^*} 

\newcommand{\sect}[1][s]{\mathtt{#1}} 

\NewDocumentCommand{\cbn}{  
  D//{\defaultlcIndex_V}
  D||{\defaultlcIndex}
}{\mathfrak{C}^{#1}_{#2}} 
\NewDocumentCommand{\Iset}{  
  D||{\defaultlcIndex}    
  e{+-}                   
  O{\defaultlclocus}      
  d()                     
}{I^{\lcIndex{#1}{#2}{#3}}_{#4}\IfNoValueF{#5}{\paren{#5}}} 

\ifcsname defineNoThmInMarionotations\endcsname
  \relax
\else 

  \newtheorem{THMprop}{Proposition}[subsection]
  \newtheorem{THMlemma}[THMprop]{Lemma}
  \newtheorem{THMthm}[THMprop]{Theorem}
  \newtheorem{THMcor}[THMprop]{Corollary}

  \newtheorem{THMconjecture}[THMprop]{Conjecture}
  \newtheorem*{THMclaim}{Claim}

  \def\makeparenother{\catcode`\(=12 \catcode`\)=12 }
  \def\makeparenactive{\catcode`\(=\active\catcode`\)=\active}

  \makeparenactive
  \NewDocumentEnvironment{textupparenenvir}{}{

    \everymath\expandafter{\makeparenother}
    \everydisplay\expandafter{\makeparenother}

    \def({\textup{\char`\(}}
    \def){\textup{\char`\)}}

    \makeparenactive
  }{\makeparenother}
  \makeparenother

  \NewDocumentEnvironment{prop}{ +o }{
    \IfNoValueTF{#1}{\begin{THMprop}}{\begin{THMprop}[{#1}]}
      \begin{textupparenenvir}
  }{
      \end{textupparenenvir}
    \end{THMprop}
  }

  \NewDocumentEnvironment{lemma}{ +o }{
    \IfNoValueTF{#1}{\begin{THMlemma}}{\begin{THMlemma}[{#1}]}
      \begin{textupparenenvir}
  }{
      \end{textupparenenvir}
    \end{THMlemma}
  }

  \NewDocumentEnvironment{thm}{ +o }{
    \IfNoValueTF{#1}{\begin{THMthm}}{\begin{THMthm}[{#1}]}
      \begin{textupparenenvir}
  }{
      \end{textupparenenvir}
    \end{THMthm}
  }

  \NewDocumentEnvironment{cor}{ +o }{
    \IfNoValueTF{#1}{\begin{THMcor}}{\begin{THMcor}[{#1}]}
      \begin{textupparenenvir}
  }{
      \end{textupparenenvir}
    \end{THMcor}
  }

  \NewDocumentEnvironment{conjecture}{ +o }{
    \IfNoValueTF{#1}{\begin{THMconjecture}}{\begin{THMconjecture}[{#1}]}
      \begin{textupparenenvir}
  }{
      \end{textupparenenvir}
    \end{THMconjecture}
  }

  \NewDocumentEnvironment{claim}{ +o }{
    \IfNoValueTF{#1}{\begin{THMclaim}}{\begin{THMclaim}[{#1}]}
      \begin{textupparenenvir}
  }{
      \end{textupparenenvir}
    \end{THMclaim}
  }

  \theoremstyle{remark}
  \newtheorem{remark}[THMprop]{Remark}

  \theoremstyle{definition}

  \numberwithin{equation}{subsection}

\fi

\allowdisplaybreaks  

\ifx\pdftexversion\undefined
  \renewcommand{\ibar}{{\raisebox{-0.9ex}{$\mathchar'26$}\mkern-6.7mu i}}
\fi

\begin{document}

\citealias{Amb03}{Ambro_quasi-log-var}
\citealias{Amb14}{Ambro_injectivity}
\citealias{Eno90}{Enoki}
\citealias{EV92}{Esnault&Viehweg_book}
\citealias{Fuj11}{Fujino_log-MMP}
\citealias{Fuj12b}{Fujino_vanishing-thms}
\citealias{Fuj13a}{Fujino_injectivity-II}
\citealias{Fuj13b}{Fujino_injectivity-hodge-theoretic}
\citealias{Fuj15b}{Fujino_survey}

\newcommand{\titlestr}{%
  An injectivity theorem on snc compact K\"ahler spaces: \\
  an application of the theory of
  harmonic integrals on log-canonical centers via adjoint ideal
  sheaves%
}

\newcommand{\shorttitlestr}{%
  An injectivity theorem on snc spaces%
}

\newcommand{\MCname}{Tsz On Mario Chan}
\newcommand{\MCnameshort}{Mario Chan}
\newcommand{\MCemail}{mariochan@pusan.ac.kr}

\newcommand{\YJname}{Young-Jun Choi}
\newcommand{\YJnameshort}{Young-Jun Choi}
\newcommand{\YJemail}{youngjun.choi@pusan.ac.kr}

\newcommand{\PNUAddressstr}{%
  Dept.~of Mathematics and Inst.~of Mathematical Science, Pusan National
  University, Busan 46241, South Korea%
}

\newcommand{\ShMname}{Shin-ichi Matsumura}
\newcommand{\ShMnameshort}{Shin-ichi Matsumura}
\newcommand{\ShMemail}{mshinichi0@gmail.com, mshinichi-math@tohoku.ac.jp}

\newcommand{\TohokuAddressstr}{%
  Mathematical Institute, Tohoku University, 6-3, Aramaki Aza-Aoba,
  Aoba-ku, Sendai 980-8578, Japan%
}

\newcommand{\subjclassstr}[1][,]{%
  32J25 (primary)#1  
  32Q15#1   
  14B05 (secondary)
}

\newcommand{\keywordstr}[1][,]{%
  $L^2$ injectivity#1
  adjoint ideal sheaf#1
  multiplier ideal sheaf#1
  log-canonical center%
}

\newcommand{\dedicatorystr}{%
}

\newcommand{\thankstr}{%
}


\title[\shorttitlestr]{\titlestr}
 
\author[\MCnameshort]{\MCname}
\email{\MCemail}

\author{\YJname}
\email{\YJemail}
\address{\PNUAddressstr}

\author{\ShMname}
\email{\ShMemail}
\address{\TohokuAddressstr}

 
\subjclass[2020]{\subjclassstr}




\let\Ker\ker
\newtheorem{step}{Step}

\def\del{\partial}
\def\we{\wedge}
\def\ov{\overline}
\newcommand{\pd}[2]{\frac{\partial#1}{\partial#2}}

\date{\today} 

\maketitle


\section{Introduction}\label{sec:intro}

{
  \let\thesubsection\thesection
  
  This paper studies an analytic aspect of higher cohomology groups of adjoint bundles
  for log-canonical (lc) pairs aiming to solve Fujino's conjecture, 
  the injectivity theorem for lc pairs on compact K\"ahler manifolds, 
  following the line of Enoki's proof. 
  This is achieved by developing the theory of harmonic integrals
  on lc centers using the analytic adjoint ideal sheaves and the
  associated residue techniques.

  The injectivity theorem, a generalization of the Kodaira vanishing theorem to semi-positive line bundles, 
  plays an important role in higher dimensional algebraic geometry. 
  After the original Koll\'ar's injectivity theorem \cite{Kollar_injectivity} had been proved 
  for semi-ample line bundles on smooth projective varieties, 
  Enoki \cite{Eno90} generalized Koll\'ar's injectivity theorem 
  to semi-positive line bundles on compact K\"ahler manifolds. 
  Koll\'ar's proof is based on theory of Hodge structures, whereas
  Enoki's proof is based on the theory of harmonic integrals, a more
  well-suited and flexible technique in the complex analytic situation. 

  Ambro and Fujino generalized Koll\'ar's theory to varieties with lc
  singularities via the theory of mixed Hodge structures,  
  motivated by applications to birational geometry (see \cite{Amb03, Amb14, EV92, Fuj11, Fuj12b, Fuj13b}). 
  It is expected that their works can also be generalized in the same line as Enoki's
  by developing an analytic treatment to lc singularities. 
  Motivated by this expectation, Fujino posed the conjecture below. 
  (Set $\ibar := \ibardefn$ \ibarfootnote\ and let $D$ be a reduced divisor for the
  rest of this article.)

  \begin{conjecture}[{Fujino's conjecture, \cite[Conjecture
      2.21]{Fuj15b}, cf.~\cite[Problem 1.8]{Fuj13a}}] 
    \label{conj:fujino}

    Let $X$ be a compact K\"ahler manifold and
    $D=\sum_{i=1}^{N}D_{i}$ be a simple-normal-crossing
    (snc) divisor on $X$.  Let $F$ be a semi-positive line bundle on
    $X$ (i.e.~it admits a smooth Hermitian metric $h_{F}$ with
    $\ibar\Theta_{h_F}(F) \geq 0$).  Consider a section
    $s \in H^{0}(X, F^{\otimes m})$ whose zero locus $s^{-1}(0)$
    contains no lc centers of the pair $(X,D)$ (i.e.~connected
    components of non-empty intersection
    $D_{i_{1}}\cap \cdots \cap D_{i_{k}}$ of the irreducible
    components $\{D_{i}\}_{i=1}^{N}$).  Then, the multiplication map
    induced by the tensor product with $s$
    \begin{equation*}
      H^{q}\paren{X, K_{X} \otimes D \otimes F}
      \xrightarrow{\otimes s} 
      H^{q}(X, K_{X} \otimes D \otimes F^{\otimes (m+1)} )
    \end{equation*}
    is injective for every $q$.
  \end{conjecture}

  The analytic theory corresponding to Koll\'ar's theory has been established for klt singularities 
  (see \cite{Cao&Demailly&Matsumura, Fujino&Matsumura, Gongyo&Matsumura,
    Matsumura_injectivity-survey, Matsumura_injectivity}).
  Therefore, it remains to develop an analytic treatment to handle the
  lc singularities.

  The cases of $\dim X=2$ and plt pairs of arbitrary dimension have been
  solved in \cite{Matsumura_injectivity-lc,
    Matsumura_rel-vanishing-w-nd} (see also \cite{Chan&Choi_injectivity-I}). 
  A full solution to Fujino's conjecture is given recently by
  Junyan Cao and Mihai P\u{a}un \cite{Cao&Paun_LC-inj}.
  In this paper, independent of the results in \cite{Cao&Paun_LC-inj},
  we prove a {\textit{generalized version}} of Fujino's conjecture  
  (Theorem \ref{thm:main}) 
  by applying the theory of harmonic integrals on lc centers of the
  given lc pair.
  Fujino's conjecture is then a direct consequence of Theorem \ref{thm:main}.

  \begin{thm}[Main Result]\label{thm:main}
    Let $X$ be a compact K\"ahler manifold  and 
    $D=\sum_{i=1}^{N}D_{i}$ be an snc divisor on $X$. 
    Let $F$ (resp.~$M$) be a line bundle on $X$ 
    with a smooth Hermitian metric $h_{F}$  (resp.~$h_{M}$) 
    such that 
    \begin{equation*}
      \ibar\Theta_{h_F}(F)\geq 0 \quad  \text{ and } \quad
      \ibar\Theta_{h_M}(M) \leq C \ibar\Theta_{h_F}(F)
      \quad \text{ for some } C>0 \; . 
    \end{equation*}
    Let $s$ be a  section of $M$  
    such that the zero locus $s^{-1}(0)$ 
    contains no lc centers of the lc pair $(X,D)$.
    Then, the multiplication map induced by the tensor product with $s$
    \begin{equation*}
      H^q(D, K_D \otimes F)
      \xrightarrow{\otimes s } 
      H^q(D, K_D \otimes F\otimes M)
    \end{equation*} 
    is injective for every $q$. 
  \end{thm}

  It can be seen from the proof that the compactness of $X$ in Theorem
  \ref{thm:main} is not necessary as soon as $D$ consists of only finitely many
  irreducible components which are compact.

  \begin{cor}[Solution to Fujino's conjecture]\label{cor:main}
    Conjecture \ref{conj:fujino} is true. 
  \end{cor}


  Our proof differs from the one in \cite{Cao&Paun_LC-inj} in the following way.
  While both works make use of (some variant of) the Hodge
  decomposition for $L^2$ forms, Cao and P\u aun prove in
  \cite{Cao&Paun_LC-inj} a Hodge decomposition for $L^2$ forms with
  respect to a K\"ahler metric with conic singularities, which induces
  a Hodge decomposition on currents (which is called the Kodaira--de
  Rham decomposition in \cite{Cao&Paun_LC-inj}) in which the Green
  kernel has controllable singularities.

  For the sake of explanation, let $u$ be an $D\otimes F$-valued
  $(n,q)$-form representing a class in $\cohgp q[X]{\logKX}$
  such that the class of $s u$ is $0$ in $\cohgp
  q[X]{\logKX M}$.
  Let also $\sect_D$ be a canonical section of $D$.
  Under our notation, the current that is under consideration in
  \cite{Cao&Paun_LC-inj} is $\frac{u}{\sect_D}$, which is not
  necessarily $L^2$ on $X$.
  Using the fact that $\eqcls{su} = 0$ in $\cohgp q[X]{\logKX M}$,
  Cao and P\u aun obtain $\frac{u}{\sect_D} =\dbar\theta + D'_{h_F}
  \beta_1 +\ibar\Theta_{h_F} \wedge \beta_2$, where $\theta$ is
  smooth while $\beta_1$ and $\beta_2$ have log-poles along
  $D+s^{-1}(0)$ (assumed to have only snc).
  It then follows from \cite{Cao&Paun_LC-inj}*{Thm.~1.1} (which
  makes use of the Hodge/Kodaira--de Rham decomposition) and the
  positivity $\ibar\Theta_{h_F} \geq 0$ that $u$ (or $u -\sect_D
  \dbar\theta$) is $\dbar$-exact.

  In our case, we make use of the residue exact sequences of adjoint
  ideal sheaves and the associated residue computation to reduce the
  setup to the union of \emph{$\sigma$-lc centers} of $(X,D)$ (i.e.~lc centers of
  codimension $\sigma$ in $X$, when $(X,D)$ is log-smooth and lc).
  Since each $\sigma$-lc center is a compact K\"ahler manifold, we
  have the Hodge decomposition (thus $L^2$ Dolbeault isomorphism and
  harmonic theory) at our disposal.
  Moreover, our reduction brings the setup essentially to the one in
  \cite{Matsumura_injectivity-lc}*{Thm.~1.6} or
  \cite{Chan&Choi_injectivity-I}*{Thm.~1.2.1} (corresponding to the
  case where $\frac{u}{\sect_D}$ is $L^2$).
  That's why we can follow the line of arguments in Enoki's proof to
  solve the conjecture via the theory of harmonic integrals on lc
  centers (and no extra resolution to bring $s^{-1}(0)$ into snc is
  needed).

  This approach gives us the advantage of obtaining Theorem
  \ref{thm:main}, a generalized version of Fujino's conjecture (see
  also Remark \ref{rem:general-commut-diagram} for other generalized
  statements which can be achieved),
  which does not seem to be derivable from results in \cite{Cao&Paun_LC-inj}, at
  least not directly. 

  Here we briefly explain the outline of the  proof of Theorem
  \ref{thm:main} with the example where the snc divisor $D$ has
  only two components $D_1$ and $D_2$ such that $D_1 \cap D_2$ is
  irreducible as an illustration.
  In this case, the union of the $1$-lc centers of $(X,D)$ is
  $\lcc|1|' = D_1 \cup D_2$ while that of the $2$-lc centers is
  $\lcc|2|' = D_1 \cap D_2$.
  For any given cohomology class $\alpha \in H^q(D,  K_{D} \otimes F)$
  such that $s  \alpha =0$ in $H^q(D,  K_{D} \otimes F \otimes M)$, 
  the goal is to show that $\alpha$ is actually $0$.

  Write $h_F = e^{-\vphi_F}$ and $h_M = e^{-\vphi_M}$, and let $\psi_D
  := \phi_D -\sm\vphi_D :=\log\abs{\sect_D}^2 -\sm\vphi_D$ be a global
  function on $X$ such that $\phi_D$ is the (local) potential (of the
  curvature of a metric) on $D$ induced from a canonical section
  $\sect_D$ and $\sm\vphi_D$ is some smooth potential on $D$.
  When $D$ is smooth (i.e.~$D_{1}\cap D_{2}=\emptyset$), 
  the class $\alpha $ can be represented by $(u_{1},  u_{2})$, where $u_i$
  is a harmonic form with respect to $\vphi_F$ on $D_i$ in
  $\mathcal{H}^{n-1,q}(D_{i}; F)_{\vphi_{F}} \cong H^{q}(D_{i},
  K_{D_{i}} \otimes F)$ for $i=1,2$.
  Enoki's argument \cite{Eno90} shows that $s u_{i}$ is also a harmonic
  form with respect to $\vphi_F +\vphi_M$ using the
  Bochner--Kodaira--Nakano formula and the given curvature assumption.
  It follows from $s \alpha =0$ (as a class) that $s u_{i}=0$ (as a form), hence
  $\alpha=(u_{1}, u_{2})=0$ as desired. 

  However, when $D =\lcc|1|'$ (as well as other $\lcc'$ in the more
  general situation) is not smooth (i.e.~$D_{1}\cap D_{2} \neq
  \emptyset$), the Dolbeault and harmonic theories for cohomology groups
  on $D$ are not yet established, obstructing the use of Enoki's
  argument.
  To overcome this difficulty, we make use of the short exact sequence
  \begin{equation*}
    \xymatrix{
      {0} \ar[r]
      & {\bigoplus_{i=1}^2 K_{D_{i}} \otimes \res F_{D_i}} \ar[r]^-{\tau}
      & {K_{D} \otimes \res F_{D}} \ar[r]
      & {K_{D_{1} \cap D_{2}} \otimes \res F_{D_1 \cap D_2}} \ar[r]
      & {0}
    } \; ,
  \end{equation*}
  where $K_{D} :=K_{X}\otimes D \otimes \frac{\holo_X}{\defidlof{D}}$
  and $\defidlof{D}$ is the defining ideal sheaf of $D$ in $X$, and its
  associated long exact sequence of cohomology groups to reduce our
  injectivity problem of the map $\otimes s$ on $D$ to the injectivity
  problems of $\otimes s$ on the lc centers of $(X,D)$ (i.e.~$D_1$,
  $D_2$ and $D_1 \cap D_2$). 
  Note that all of the lc centers are not contained in $s^{-1}(0)$ by
  assumption and are compact K\"ahler manifolds on which the Dolbeault
  isomorphism and harmonic theory are available.

  Such strategy is suggested already in \cite{Matsumura_injectivity-lc}
  and is used there in the proof of the injectivity theorem for plt
  pairs.
  It is framed in \cite{Chan&Choi_injectivity-I} in terms of the adjoint
  ideal sheaves $\aidlof* := \aidlof = \mtidlof<X>{\vphi_F} \cdot
  \defidlof{\lcc+1'} = \defidlof{\lcc+1'}$ (the defining ideal sheaf
  of $\lcc+1'$ in $X$, under the assumption $\vphi_F$ being smooth)
  for integers $\sigma \geq 0$ and the corresponding residue morphisms
  $\Res^\sigma$ for $\sigma \geq 1$ (see Section \ref{subsec:residue}
  for the definitions).
  Writing $\lcc' = \bigcup_{p \in \Iset} \lcS$ as the decomposition of
  $\lcc'$ into the (irreducible) $\sigma$-lc centers $\lcS$, the residue
  morphism $\Res^\sigma$ induces the isomorphism
  \begin{equation*}
    \logKX \otimes \faidlof/-1* \xrightarrow[\isom]{\Res^\sigma}
    \logKX \otimes \residlof* := \bigoplus_{p \in \Iset} K_{\lcS}
    \otimes \res F_{\lcS} 
  \end{equation*}
  (notice that $\logKX \otimes \frac{\defidlof{D_1 \cap
      D_2}}{\defidlof{D}} \isom \bigoplus_{i=1}^2 K_{D_{i}} \otimes \res
  F_{D_i}$ and $\logKX \otimes \frac{\holo_X}{\defidlof{D_1 \cap D_2}}
  \isom K_{D_{1} \cap D_{2}} \otimes \res F_{D_1 \cap D_2}$ in the
  example).
  It can then be seen that, for more general $D$, the reduction can be
  done via the short exact sequences $0 \to \faidlof/-1* \to
  \faidlof|\sigma'|/-1* \to \faidlof|\sigma'|/* \to 0$ for some integers
  $\sigma$ and $\sigma'$ such that $1 \leq \sigma \leq \sigma'$,
  together with an induction on $\sigma$ via some diagram-chasing
  argument.
  See Step \ref{step:harmonic-rep} of Section
  \ref{sec:proof-of-simple-case} and the beginning of Section
  \ref{subsec:general} for precise details.

  After the reduction, we are led to consider the maps
  \begin{equation*}
    \renewcommand{\objectstyle}{\displaystyle}
    \xymatrix{
      {\smash{\bigoplus_{i =1}^2}\:\cohgp q[D_i]{K_{D_i} \otimes
          F}} \ar[r]^-{\tau}
      \ar[dr]^-{\nu}
      &{\cohgp q[D]{K_D \otimes F}}
      \ar[d]^-{\otimes s} \ar@{}@<-1em>[d]_*+{\circlearrowright}
      \\
      &{\cohgp q[D]{K_D \otimes F \otimes M} \; .}
    }
  \end{equation*}
  It suffices to prove that $\ker\nu =\ker\tau$ (Theorem
  \ref{thm:ker-nu=ker-tau}).
  Indeed, given the injectivity of the map $\otimes \res s_{D_1 \cap D_2}$ on
  $\cohgp q[D_1 \cap D_2]{K_{D_1 \cap D_2} \otimes F}$ followed from
  Enoki's argument in the previous case, we see that the given class
  $\alpha \in \cohgp q[D]{K_D \otimes F}$ actually lies in the image
  $\im\tau$ of $\tau$, say, $\alpha = \tau\paren{u_1, u_2}$ for some
  harmonic forms $u_i \in \Harm'/n-1,q/<D_i>{F},{\vphi_F} \isom \cohgp
  q[D_i]{K_{D_i} \otimes F}$.
  It will then follow that $\paren{u_1, u_2} \in \ker\nu
  =\ker\tau$, hence $\alpha =0$, as desired.
  The pair $(u_1,u_2)$ can be treated as a representative of $\alpha$.
  Suggested by the fact that a harmonic form is the unique
  representative with the \emph{minimal} $L^2$ norm among all elements
  in its corresponding $L^2$ Dolbeault cohomology class, we can choose
  an ``optimal'' representative of $\alpha$ such that $(u_1, u_2)$ has
  the \emph{minimal} distance from (i.e.~is orthogonal to) the subspace
  $\ker\tau$ with respect to the $L^2$ norm induced from $\vphi_F$.
  It then suffices to show that $u_i =0$ for $i = 1,2$ to prove that
  $\ker\nu =\ker\tau$.
  This is done by following the proof of
  \cite{Matsumura_injectivity-lc}*{Thm.~1.6} or
  \cite{Chan&Choi_injectivity-I}*{Thm.~1.2.1} (therefore following the
  spirit of Enoki's argument), but with a few technical modifications.

  One technical complication comes from the use of \v Cech cohomology
  for some cohomology groups (e.g.~$\cohgp q[D]{K_D \otimes F}$) due to
  the lack of the Dolbeault isomorphism.
  Another one is that the argument of Takegoshi in
  \cite{Chan&Choi_injectivity-I}*{\S 3.1, Step IV} (see also
  \cite{Matsumura_injectivity-lc}*{Prop.~3.13}), which essentially gives
  rise to an element in $\ker\tau$ constructed from $u_i$'s, is replaced
  by a construction of a harmonic forms $w$ (or a collection $w
  :=\paren{w_b}_{b\in \Iset+1}$ of harmonic forms for the general $D$)
  representing a class in $\cohgp{q-1}[D_1 \cap D_2]{K_{D_1 \cap D_2}
    \otimes F}$ (see \eqref{eq:w-prelim-formula} and \eqref{eq-def-w}).
  The class of $w$ has its image lying in $\ker\tau$ via the connecting
  morphism of the relevant long exact sequence.
  Such construction is suggested by a residue computation, which relates
  an inner product on (the normalization of) $\lcc|1|'$ to an inner
  product on (lower dimensional) $\lcc|2|'$ (see Proposition
  \ref{prop:res-formula-dbar-exact-dot-harmonic}; see also Steps
  \ref{item:express-su-in-residue-norm} and \ref{item:pf:use_u-ortho-w}
  in Section \ref{subsec:general}, or Steps
  \ref{item:expression-of-su-simple} and
  \ref{step:pf:use_u-ortho-w-simple} in Section
  \ref{sec:proof-of-simple-case} for less intensive notation).
  Such relation between the inner products shows that $w$ is the
  obstruction for having $u_i = 0$ for $i=1,2$.
  This becomes the crucial ingredient to complete the proof.

  The proof of Theorem \ref{thm:main} for the case of general $D$
  follows the same arguments.
  A brief comment for the case where $\vphi_F$ and $\vphi_M$ possess
  suitable analytic singularities is given in Remarks
  \ref{rem:singular-vphi_F} and \ref{rem:no-hard-Lefschetz}.

}

This paper is organized as follows.
\tableofcontents

\subsection*{Acknowledgments}
The authors would like to thank the members of Bayreuth University and Pusan National University for their hospitality.
This paper is resulted from the discussions there. 
S.M.~would like to thank Professors Junyan Cao and Mihai P\u{a}un for sharing a preliminary version of \cite{Cao&Paun_LC-inj}.
Also, he would like to thank Professor Osamu Fujino 
for his encouragement and long-standing discussions on lc singularities. 
He is partially supported 
by Grant-in-Aid for Scientific Research (B) $\sharp$21H00976 
and Fostering Joint International Research (A) $\sharp$19KK0342 from
JSPS.
Y.C.~and M.C.~would like to thank S.M.~for drawing their attention to
Fujino's conjecture not long before the covid pandemic (which results
in \cite{Chan&Choi_injectivity-I}) and for joining hand to complete
this project when most aspects of life went back to normal.
Y.C.~and M.C.~were supported by the National Research Foundation
of Korea (NRF) Grant funded by the Korean government
(Nos.~2023R1A2C1007227 and 2021R1A4A1032418).

\section{Preliminary results}\label{sec:preliminaries}

\subsection{Notation and conventions}\label{subsec:notation}



The following notions are used throughout this paper unless stated otherwise. 
\begin{itemize}
\item $(X,\omega)$ is a compact K\"ahler manifold of dimension $n$. 


\item $h_F := e^{-\vphi_F}$ and $h_M := e^{-\vphi_M}$, where $\vphi_F$ and
  $\vphi_M$ are respectively the given potentials on $F$ and $M$.
  
\item $D=\sum_{i \in \Iset||}D_{i}$ is a reduced simple-normal-crossing (snc)
  divisor on $X$ (where $\Iset||$ is a finite set). 

\item $\sect_i$ is a canonical section  of the irreducible component $D_{i}$. 

\item $\sect_D := \prod_{i\in \Iset||} \sect_i$ is the canonical section of $D$. 

\item $\sigma \in \{0,1,2,\cdots, n\}$.



\item $\lcc' :=\bigcup_{p \in \Iset} \lcS$ is the union of
  \emph{$\sigma$-lc centers of $(X,D)$}, i.e.~the
  $\sigma$-codimensional irreducible components of any intersections
  of irreducible components of $D$ (under the assumption $(X,D)$ being
  log-smooth and lc), indexed by $\Iset$.
  Set $\lcc|0|' := X$ and let $\Iset|0|$ be a singleton for convenience.
  Note also that $\Iset|1| = \Iset||$.

\item $\Diff_{p}D$ is the effective divisor on $\lcS$ defined by the 
adjunction formula 
\begin{equation*}
  K_{\lcS} \otimes \Diff_{p}D = \parres{K_X \otimes D}_{\lcS}
\end{equation*}
such that the restriction of $\sect_{(p)}:=
\smashoperator{\prod\limits_{i \in \Iset|| \colon D_i
    \not\supset \lcS}} \sect_i $ to $\lcS$ is a canonical section of
$\Diff_{p}D$.

\item $\phi_D :=\log\abs{\sect_D}^2$ and $\phi_{(p)}
  :=\log\abs{\sect_{(p)}}^2$ are the potentials induced from the
  canonical sections of $D$ and $\Diff_p D$.

\item $\cvr V := \{V_{i}\}_{i \in I}$ is an open cover of $X$  by admissible open sets. 

\item $\{\rho^{i}\}_{i\in I}$ is a partition of unity subordinate to
  $\cvr V$. 
\end{itemize}

Here an open set $V \subset X$ is said to be \emph{admissible} with
respect to $D$ if $V$ is biholomorphic to a polydisc centered at the
origin under a holomorphic coordinate system $(z_{1}, z_{2}, \cdots,
z_{n})$ such that
\begin{equation*} 
  D =\set{z_1 \dotsm z_{\sigma_V} =0}, \quad 
  \log r_{j}^2 < 0, \quad \text{and }
  r_j \fdiff{r_j} \psi_D >0 \text{ on } V \; , 
\end{equation*} 
where  $r_j := \abs{z_j}$  and $\res{\psi_D}_V := \parres{\phi_D
  -\sm\vphi_D}_V =\sum_{j=1}^{\sigma_V} \log\abs{z_j}^2
-\res{\sm\vphi_D}_V$. 

When an admissible set $V$ is considered, an index $p \in \Iset$ such
that $\lcS \cap V \neq \emptyset$ is interpreted as a permutation
representing a choice of $\sigma$ elements from the set
$\set{1,2,\dots,\sigma_V}$ such that
\begin{equation*}
  \lcS \cap V = \set{z_{p(1)} = z_{p(2)} = \dotsm = z_{p(\sigma)} = 0}
  \quad\text{ and }\quad
  \res{\sect_{(p)}}_V = z_{p(\sigma+1)} \dotsm z_{p(\sigma_V)}
\end{equation*}
(cf.~the definition of the set $\cbn$ in \cite{Chan_adjoint-ideal-nas}*{\S 3.1}).



\subsection{$L^{2}$ Dolbeault isomorphism and some results on harmonic
forms}\label{subsec:l2}


{
  \setDefaultvphi{\vphi_L}

  Let $L$ be a holomorphic line bundle on $X$ equipped with a
  (possibly singular) quasi-psh potential $\vphi_L$, which induces,
  together with the K\"ahler form $\omega$, an $L^2$ norm
  $\norm{\cdot}_{X} := \norm\cdot_{X,\vphi_L,\omega}$ on the space of
  smooth $K_X \otimes L$-valued $(0,q)$-forms (or $L$-valued
  $(n,q)$-forms) on $X$.
  Let $\Ltwo/n,q/{L}_{\vphi_L}$ be the completion with respect to $\norm\cdot_X$
  and $\Harm :=\Harm{L}$ be the space of harmonic forms with respect to $\norm\cdot_X$.
  The $L^2$ Dolbeault isomorphism (see
  \cite{Matsumura_injectivity}*{Prop.~5.5 and 5.8} and
  \cite{Matsumura_injectivity-lc}*{Prop.~2.8} for a proof, and see 
  \cite{Chan&Choi_injectivity-I}*{footnote 1} for its naming) guarantees
  the closedness of the subspaces in the orthogonal decomposition
  \begin{equation*}
    \Ltwo/n,q/{L}_{\vphi_L}
    = \Harm \oplus \cl{\paren{\im\dbar}}_{\vphi_L} \oplus \cl{\paren{\im\dbadj}}_{\vphi_L}
    = \Harm \oplus \paren{\im\dbar}_{\vphi_L} \oplus \paren{\im\dbadj}_{\vphi_L} 
  \end{equation*}
  (where $\dbadj$ is the Hilbert space adjoint of $\dbar$ with respect
  to $\norm\cdot_X$, $\paren{\im\dbar}_{\vphi_L}$ and
  $\paren{\im\dbadj}_{\vphi_L}$ denote the images of the corresponding
  operators, with $\cl{\paren{\im\dbar}}_{\vphi_L}$ and
  $\cl{\paren{\im\dbadj}}_{\vphi_L}$ being their closures in
  $\Ltwo/n,q/{L}_{\vphi_L}$)
  and the isomorphism
  \begin{equation*}
    \Harm \isom \cohgp q[X]{K_X \otimes L \otimes \mtidlof{\vphi_L}}
  \end{equation*}
  between the space of harmonic forms and the \v Cech cohomology
  group.
  With $\cvr V := \set{V_i}_{i\in I}$ and $\set{\rho^i}_{i\in I}$ given in Section \ref{subsec:notation},
  the isomorphism can be given explicitly as follows.
  For any (alternating) \v Cech $q$-cocycle $\set{\alpha_{\idx 0.q}}_{\idx 0,q \in
    I}$ and any harmonic form $u \in \Harm$ such that they represent
  the same class in $\cohgp q[X]{K_X \otimes L \otimes
    \mtidlof{\vphi_L}}$, the two representatives are related by 
  (under the Einstein summation convention)
  \begin{equation} \label{eq:Cech-Dolbeault-isom}
    \begin{aligned}
      u &=\dbar v_{(2)} +\dbar \rho^{i_{q-1}} \wedge \dotsm \wedge
      \dbar\rho^{i_0} \alpha_{\idx 0.q} \qquad\paren{\forall~ i_q \in
        I}
      \\
      &=\dbar v_{(2)} +\dbar \rho^{i_{q-1}} \wedge \dotsm \wedge
      \dbar\rho^{i_0} \cdot \rho^{i_q} \:\alpha_{\idx 0.q}
      \\
      &=\dbar v_{(2)} +(-1)^q \:\underbrace{\dbar \rho^{i_{q}} \wedge
        \dotsm \wedge \dbar\rho^{i_1} \cdot \rho^{i_0} }_{=: \:
        \paren{\dbar\rho}^{\idx q.0}} \alpha_{\idx 0.q}
    \end{aligned}
  \end{equation}
  for some $K_X \otimes L$-valued $(0,q-1)$-form $v_{(2)}$ on $X$ with
  $L^2$ coefficients with respect to $\norm\cdot_{X}$ (see
  \cite{Matsumura_injectivity}*{Prop.~5.5} or
  \cite{Chan&Choi_injectivity-I}*{Lemma 3.2.1}).

  The above result is applicable also to the case when $L$ is replaced by
  $D \otimes L$ equipped with the potential $\phi_D +\vphi_L$, where
  $\phi_D :=\log\abs{\sect_D}^2$.
  Denote the corresponding $L^2$ norm by $\norm\cdot_{X,\phi_D}$.
  Assume that \emph{$\vphi_L$ is smooth on $X$}.
  We state the following simple fact here for clarity.
  \begin{lemma} \label{lem:su-harmonicity}
    If $u \in \Harm{L}$, then $\sect_D u \in \Harm{D\otimes L},{\phi_D+\vphi_L}$.
  \end{lemma}
  
  \begin{proof}
    Since $\sect_D$ is holomorphic, it is clear that $\sect_D u$ is
    $\dbar$-closed.

    Let $\dfadj$ and $\dfadj_{\phi_D}$ be the formal adjoint of
    $\dbar$ with respect to $\vphi_L$ and $\phi_D +\vphi_L$
    respectively.
    It then follows that $\dfadj_{\phi_D} = \dfadj
    +\idxup{\diff\phi_D} . \cdot$ and 
    \begin{equation*}
      \dfadj_{\phi_D} \paren{\sect_D u}
      = \sect_D \:\dfadj u - \idxup{\diff\sect_D}. u
      +\idxup{\diff\phi_D} .\sect_D u
      =\sect_D \dfadj u = 0 \; .
    \end{equation*}
    Note that $\omega$ is not complete on $X \setminus D$ and the
    claim (in particular, $\sect_D u \in \Dom \dbadj_{\phi_D}$, where
    $\dbadj_{\phi_D}$ is the Hilbert space adjoint of $\dbar$ with
    respect to $\norm\cdot_{X,\phi_D}$) cannot follow from the
    standard result (for example, \cite{Demailly}*{Ch.~VIII,
      Thm.~(3.2c)}).
    Indeed, the proof of $su \in \Dom
    \dbadj_{\vphi_M}$ in \cite{Chan&Choi_injectivity-I}*{Cor.~3.2.6}
    gives precisely the result $\sect_D u \in \Dom\dbadj_{\phi_D}$ in
    the current setting, which completes the proof.
    A sketch of it is given below for readers' convenience.
    
    Let $\theta \colon [0,\infty) \to [0,1]$ be a smooth
    non-decreasing cut-off function such that
    $\res\theta_{[0,\frac12]} \equiv 0$ and $\res\theta_{[1,\infty)}
    \equiv 1$.
    Set $\theta_\eps := \theta \circ \frac{1}{\abs{\psi_D}^\eps}$ and
    $\theta'_\eps := \theta' \circ \frac{1}{\abs{\psi_D}^\eps}$ for
    every $\eps \geq 0$ (where $\theta'$ is the derivative of
    $\theta$).
    Then both $\theta_\eps$ and $\theta'_\eps$ have compact supports
    inside $X \setminus D$ for $\eps > 0$ and $\theta_\eps \ascendsto
    1$ pointwisely on $X \setminus D$ as $\eps \descendsto 0$.
    For any $\zeta \in \Dom\dbar \subset \Ltwo/n,q-1/<X>{D\otimes
      L}_{\phi_D+\vphi_L}$, convolution with a smoothing kernel on
    local coordinate charts and the lemma of Friedrichs guarantees the
    existence of a sequence $\seq{\zeta_{\eps, \nu}}_{\nu\in\Nnum}$ of
    smooth forms compactly supported in $X \setminus D$ such that
    $\zeta_{\eps,\nu} \tendsto \theta_\eps \zeta$ in the graph norm
    $\paren{\norm\cdot_{X,\phi_D}^2
      +\norm{\dbar\:\cdot}_{X,\phi_D}^2}^{\frac 12}$ of $\dbar$ for
    each $\eps > 0$.
    It then follows that
    \begin{align*}
      \iinner{\sect_D u}{\dbar\zeta}_{X,\phi_D} 
      \xleftarrow{\eps \tendsto 0^+}
      &~\iinner{\sect_D u}{\theta_\eps \dbar\zeta}_{X,\phi_D} \\
      =&~\iinner{\sect_D u}{\dbar\paren{\theta_{\eps}\zeta}}_{X,\phi_D}
         -\iinner{\sect_D u}{\dbar\theta_\eps \wedge \zeta}_{X,\phi_D} \\
      \xleftarrow{\nu \tendsto \infty}
      &~\iinner{\sect_D u}{\dbar\zeta_{\eps,\nu}}_{X,\phi_D}
        -\iinner{\sect_D u}{\frac{\eps \theta'_\eps}{\abs{\psi_D}^{1+\eps}}
        \dbar\psi_D \wedge \zeta}_{X,\phi_D} \\
      =&~\iinner{\dfadj_{\phi_D} \paren{\sect_D u}}{\zeta_{\eps,\nu}}_{X,\phi_D}
         -\iinner{\frac{\eps \theta'_\eps}{\abs{\psi_D}^{1+\eps}}
         \idxup{\diff\psi_D} . \sect_D u}{\zeta}_{X,\phi_D} \; .
    \end{align*}
    The inner product on the far right-hand-side converges to $0$ as
    $\eps \tendsto 0^+$, a consequence of the residue computation (see
    \cite{Chan&Choi_injectivity-I}*{Prop.~3.2.3 and Remark 3.2.4}).
    We can then conclude that $\sect_D u \in \Dom\dbadj_{\phi_D}$ after
    letting $\nu \tendsto \infty$ and then $\eps \tendsto 0^+$.
  \end{proof}

}

Now consider the cases where $(L, \vphi_L) =(F, \vphi_F)$
and $(L, \vphi_L) =(F\otimes M, \vphi_F +\vphi_M)$.
A consequence of the positivity on $F$ and $M$ in
\cite{Enoki}, \cite{Matsumura_injectivity-lc} and
\cite{Chan&Choi_injectivity-I} are recalled below.
\begin{prop} \label{prop:consequence-of-positivity}
  Suppose that $\vphi_F$ is smooth such that
  $\ibddbar\vphi_F \geq 0$ and $u \in \Harm{F}$.
  Then, one has  $\nabla^{(0,1)}u = 0$.
  If, furthermore, $\vphi_M$ is smooth and satisfies
  $\paren{- C \omega \leq} \; \ibddbar\vphi_M \leq C\ibddbar\vphi_F$
  for some constant $C > 0$,
  then one also has $su \in \Harm{F\otimes M},{\vphi_F+\vphi_M}$.
\end{prop}

\begin{proof}[Reference to the proof]
  These results follow directly from the Bochner--Kodaira--Nakano
  formula.
  See \cite{Chan&Choi_injectivity-I}*{Prop.~3.2.5 and
    Cor.~3.2.6} (while taking $D=0$ and $\psi_D \equiv -1$ in those
  statements).
  See also the proofs for $\diff^*_h\xi = 0$ in
  \cite{Enoki}*{Prop.~2.1} or $D'^*u = 0$ in
  \cite{Matsumura_injectivity-lc}*{Prop.~3.7}.
  These are equivalent statements to the claim $\nabla^{(0,1)}u =0$
  (indeed, $\diff^*_h = D'^*$ and $\abs{D'^*u}^2 =
  \abs{\nabla^{(0,1)}u}^2$ by \cite{Chan&Choi_injectivity-I}*{Remark
    2.4.3}).
\end{proof}

Lemma \ref{lem:su-harmonicity} and Proposition
\ref{prop:consequence-of-positivity} are applied to the case with
$\lcS$ in place of $X$ and $\phi_{(p)}$ in place of $\phi_D$ in the
following sections.


\subsection{Adjoint ideal sheaves and the residue computations}
\label{subsec:residue}


{
  \setDefaultvphi{\vphi_L}

  Let $L$ be a line bundle on $X$ equipped with a \mmark{smooth metric
    $e^{-\vphi_L}$}{$\vphi_L$ has to be smooth, or the claim on the
    jumping number must be mentioned explicitly. The result  $\aidlof*
    =\mtidlof{\vphi_L} \cdot \defidlof{\lcc+1'}$ may not hold
    otherwise. \\ }. 
  The \mmark[BlueGreen]{residue function $\eps \mapsto \RTF|f|,<V>$}{It's
    possible not using ``$\RTF|f|$'' at all in this paper.} of index $\sigma$ 
  is defined, for each \mhlight[BlueViolet]{$f \in  \logKX[L] \otimes
    \smooth_X (V)$}, to be 
  \begin{equation*}
    \RTF|f|,<V> :=\RTF|f|,<V>,
    := \eps \int_V \frac{\abs f^2 \:e^{-\phi_D-\vphi_L}}{\logpole} \quad
    \text{ for } \eps > 0  \; . 
  \end{equation*}\mariocomment[BlueViolet]{For consistency of notation in this
    section only.}%
  The adjoint ideal sheaf $\aidlof :=\aidlof<X>$ of index $\sigma$
  is given at each $x \in X$  by
  \begin{equation*}
    \aidlof_x :=\setd{f \in \holo_{X,x}}{
      \exists~\text{open set } V_x \ni x \:, \: \forall~\eps > 0 \:, \:
      \RTF|f|,<V_x>, < +\infty
    } \; .
  \end{equation*}
  Note that the adjoint ideal sheaf is independent of $\vphi_L$ (as
  $\vphi_L$ is smooth).
  By \cite{Chan_adjoint-ideal-nas}*{Thm.~1.2.3}, the adjoint ideal
  sheaf can be written as 
  \begin{equation*}
    \aidlof = \mtidlof{\vphi_L} \cdot \defidlof{\lcc+1'}
    =\defidlof{\lcc+1'}
    \quad\text{ for any } \sigma \geq 0 \; ,
  \end{equation*}
  where $\defidlof{\lcc+1'}$ is the defining ideal sheaf of $\lcc+1'$
  in $X$ (with the reduced structure), \mmark{and we have the residue short exact
    sequence}{I don't want to suggest that the product structure of
    $\aidlof*$ implies directly the residue exact sequence.}
  \begin{equation*}
    \xymatrix@R-0.5cm@C+0.3cm{
      {0} \ar[r]
      & {\aidlof-1} \ar[r]
      & {\aidlof} \ar[r]^-{\Res^\sigma}
      & {\residlof} \ar[r]
      & {0 \; .}
    }
  \end{equation*}
  Here the quotient sheaf ${\residlof}$, called the \emph{residue sheaf of index $\sigma$}, can be written as 
  \begin{equation*}
    \residlof
    = \bigoplus_{p \in \Iset} \paren{\Diff_p D}^{-1}
    \otimes \mtidlof<\lcS>{\vphi_L}
    = \bigoplus_{p \in \Iset} \paren{\Diff_p D}^{-1}
  \end{equation*}
  Note $\logKX[L] \otimes \residlof =\bigoplus_{p \in\Iset} K_{\lcS} \otimes \res L_{\lcS}.$
  Next we describe the \emph{residue morphism $\Res^\sigma$} in terms of 
  the Poincar\'e residue map $\PRes[\lcS]$ given in
  \cite{Kollar_Sing-of-MMP}*{\S 4.18} as follows. 
  The Poincar\'e residue map $\PRes[\lcS]$ from $X$ to each $\lcS$ is
  uniquely determined after an orientation on the conormal bundle of
  $\lcS$ in $X$ is fixed.
  For an admissible open set $V \subset X$, 
  we have $\lcc' \cap V = \bigcup_{\alert{p \in \Iset}}
  \lcS<V>$ \mmark{(where $\lcS<V> := \lcS \cap V$, which is connected by the
  definition of the admissible open set, and possibly empty)}{This is
  a subtle fact that is used in the residue computation. We can keep
  using the same index set because $V$ is admissible.} and $\lcS<V>
=\set{z_{p(1)} =z_{p(2)} =\dotsm =z_{p(\sigma)}=0}$ when non-empty. 
  Under such coordinate system, a section $f $ of  $\logKX[L] \otimes
  \aidlof$ on $V \subset X$ can be written as
  \begin{equation*}
    f = \;\;\smashoperator{\sum_{p \in \Iset \colon \lcS<V>
        \neq\emptyset}} \;\; dz_{p(1)} \wedge \dotsm \wedge dz_{p(\sigma)}
    \wedge g_p \:\sect_{(p)} 
    =\;\;\smashoperator[l]{\sum_{p \in \Iset \colon \lcS<V>
        \neq\emptyset}}
    \frac{dz_{p(1)}}{z_{p(1)}} \wedge \dotsm
    \wedge \frac{dz_{p(\sigma)}}{z_{p(\sigma)}}
    \wedge g_p \:\sect_D \quad\text{ on } V. 
  \end{equation*}
  \mmark{We therefore see that 
  \begin{equation*}
    \PRes[\lcS](\frac{f}{\sect_D})  =\res{g_p}_{\lcS} \in
    K_{\lcS} \otimes \res L_{\lcS} \quad\text{ on } \lcS<V> 
  \end{equation*}}{I don't want to give the impression that we define
  the Poincar\'e residue map by this formula.}%
  under the assumption that the orientation on the conormal bundle of
  $\lcS$ in $X$ on $V$ is given by $(dz_{p(1)}, dz_{p(2)}, \dots,
  dz_{p(\sigma)})$. 
  On the other hand, the residue morphism $\Res^\sigma$ is given in
  \cite{Chan_adjoint-ideal-nas}*{\S 4.2} by 
  \begin{equation*}
    \renewcommand{\objectstyle}{\displaystyle}
    \xymatrix@C+0.5cm@R-0.5cm{
      {\logKX[L] \otimes \aidlof} \ar[r]^-{\Res^\sigma}
      \ar@{}[d]|*[left]+{\in} 
      & {\hphantom{\logKX[L] \otimes \residlof}}
      \save +<4em,-1.3ex>*{\logKX[L] \otimes \residlof
        =\bigoplus_{p \in\Iset} K_{\lcS} \otimes \res L_{\lcS}} \restore
      \ar@{}[d]|*[left]+{\in}
      \\
      *+<0.8cm,0cm>{f} \ar@{|->}[r]
      & {\paren{\res{g_p}_{\lcS}}_{\mathrlap{p\in\Iset}}
        \mathrlap{\hphantom{p\in\Iset} .}} 
    }
  \end{equation*}
  Assuming $f$ being defined on a neighbourhood $V'$ of the closure
  $\cl V$ of $V$ and letting $\rho \colon V' \to [0,1]$ be a compactly
  supported smooth function
  identically equal to $1$ on $V$, one obtains, 
  from the result in \cite{Chan_adjoint-ideal-nas}*{Thm.~4.1.2 (2)} (or the
  computation in \cite{Chan_on-L2-ext-with-lc-measures}*{Prop.~2.2.1}
  or \cite{Chan&Choi_ext-with-lcv-codim-1}*{Prop.~2.2.1}),
  a (squared) norm of $g :=\paren{\res{g_p}_{\lcS<V>}}_{p \in \Iset} \in
  \logKX[L] \otimes \residlof$ on $V$ given by
  \begin{equation} \label{eq:residue-norm}
    \norm{g}_{\lcc<V>'}^2 :=\RTF|f|(0),<V>
    =\lim_{\rho \descendsto \charfct_{\cl V}} \lim_{\eps \tendsto 0^+}
    \RTF[\rho]|f|,<V'>
    =\sum_{p \in \Iset} \frac{\pi^\sigma}{(\sigma -1)!}
    \int_{\mathrlap{\lcS<V>}} \;\;\;
    \abs{g_p}^2 \:e^{-\vphi_L}
    =:\sum_{p\in\Iset} \norm{g_p}_{\lcS<V>}^2 \; ,
  \end{equation}
  where the limit $\lim_{\rho \descendsto \charfct_{\cl V}}$ refers to
  the pointwise limit as $\rho$ descends to the characteristic
  function $\charfct_{\cl V}$ of $\cl V$ on $X$.
  Such a norm is referred to as the \emph*{residue norm on $\logKX[L]
    \otimes \residlof$ on $V$}.
  Moreover, we also see from the residue exact sequence that
  \begin{equation*}
    \aidlof-1_x
    =\setd{f \in \aidlof_x}{ \exists~\text{open set } V_x
      \ni x \:, \: \RTF|f|(0),<V_x> = 0}
  \end{equation*}
  for every $x \in X$.

  Under the assumption that $\vphi_L$ has only neat analytic
  singularities (which is indeed smooth in the current setting), the
  residue norm on an admissible open set $V \subset X$ can also be
  obtained from 
  \begin{equation*}
    \lim_{\eps \tendsto 0^+} \eps \int_{V} \frac{
      \rho \abs f^2 \:e^{-\phi_D-\vphi_L}
    }{\abs{\psi_D}^{\sigma +\eps}}
    =\RTF[\rho]|f|(0),<V>
  \end{equation*}
  for any smooth compactly supported cut-off function $\rho$ on $V$ (see
  \cite{Chan&Choi_ext-with-lcv-codim-1}*{Prop.~2.2.1} or
  \cite{Chan&Choi_injectivity-I}*{Thm.~2.6.1}).
  Moreover, the above equation works not only for $f$ with holomorphic
  coefficients, but also for $f$ with coefficients in
  $\smooth_{X\,*}$, where
  \begin{align*}
    \smooth_{X\, *}
    &:=\paren{\smooth_{X}\left[
      \frac{1}{\abs{\sect_i}} \colon i \in \Iset||
      \right]}_{\text{b}}
      \qquad\paren{\sect_i \text{ treated as a local defining function of }
      D_i} \\
    &:=\set{\text{locally bounded elements in the $\smooth_X$-algebra generated
      by } \frac{1}{\abs{\sect_i}} \text{ for all } i\in\Iset||} \;
      .\footnotemark
  \end{align*}%
  \footnotetext{
    On an admissible open set $V$ under the holomorphic coordinate
    system $(z_1,\dots, z_n)$ such that $D\cap V =\set{z_1 z_2 \dotsm
      z_{\sigma_V} =0}$, one has
    \begin{equation*}
      \smooth_{X \,*}(V)
      =\smooth_X(V)\left[e^{\pm \cplxi \theta_1}, \dots, e^{\pm \cplxi
          \theta_{\sigma_V}} \right]
    \end{equation*}
    where $(r_j,\theta_j)$ is the polar coordinate system of the
    $z_j$-plane for $j=1,\dots,\sigma_V$ in $V$, which is (almost) the
    same as the ad hoc definition of $\smooth_{X\, *}(V)$ given in 
    \cite{Chan&Choi_injectivity-I}*{\S 2.6} (in which
    $e^{\pm\cplxi\theta_{k}}$ for $k \geq \sigma_V +1$ are also included
    in the set of generators of the algebra).
    The definition given here is independent of coordinates and its
    sheaf structure can be seen easily.
  }%
  The coefficients of $\Res^\sigma$ (and hence $\PRes[\lcS]$ for any
  $p\in\Iset$) can be extended from $\holo_X$ to $\smooth_{X\,*}$
  accordingly.
  The residue norm is finite when the coefficients of $f$ belong to
  $\smooth_{X\,*} \cdot \aidlof$ on $V$.
  When the induced inner product is considered, one still has
  finiteness even if one of the argument does not have coefficients in
  $\smooth_{X\,*} \cdot \aidlof$, which is the content of the
  following proposition.
  \begin{prop} \label{prop:residue-product-X-to-lcS}
    Given any admissible open set $V \subset X$ and any section $f \in
    \logKX[L] \otimes \smooth_{X \:c\,*} \cdot\aidlof\paren{V}$
    (compactly supported in $V$) such
    that $\Res^\sigma(f) = g =\paren{g_p}_{p\in\Iset}$, one
    has, for any $\xi \in \logKX[L] \otimes \smooth_{X \, *}\paren{V}$,
    \begin{align*}
      \lim_{\eps \tendsto 0^+} \eps \int_V
      \frac{\inner{\xi}{f} \:e^{-\phi_D-\vphi_L}}{\abs{\psi_D}^{\sigma
      +\eps}}
      &=\sum_{p \in \Iset} \frac{\pi^\sigma}{(\sigma-1)!}
        \int_{\lcS<V>} \inner{\frac{\rs*\xi_p}{\sect_{(p)}}}{\: g_p}
        \:e^{-\vphi_L} \\
      &=\sum_{p \in \Iset}
        \underbrace{
        \frac{\pi^\sigma}{(\sigma-1)!}
        \int_{\lcS<V>} \inner{\rs*\xi_p}{\: g_p \sect_{(p)}}
        \:e^{-\phi_{(p)}-\vphi_L}
        }_{\displaystyle =:
        \iinner{\rs*\xi_p}{g_p\sect_{(p)}}_{\mathrlap{\lcS<V>,
        \phi_{(p)}}}}
    \end{align*}
    which is finite,
    where $\phi_{(p)} :=\log\abs{\sect_{(p)}}^2$ and
    \begin{equation*}
      \rs*\xi_p := \PRes[\lcS](\frac{\xi}{\sect_D}) \cdot \sect_{(p)}
      \in K_{\lcS} \otimes \Diff_p D \otimes \res L_{\lcS} \otimes
      \smooth_{X\:c\, *}\paren{\lcS<V>} \; .
    \end{equation*}
    Moreover, if either $f$ or $\xi$ belongs to $\logKX[L] \otimes
    \smooth_{X\,*} \cdot\aidlof-1\paren{V}$, then $\eps \int_V
      \frac{\inner{\xi}{f} \:e^{-\phi_D-\vphi_L}}{\abs{\psi_D}^{\sigma
      +\eps}} = \BigO(\eps)$ (the big-O notation) as $\eps \tendsto 0$.
  \end{prop}

  \begin{proof}
    By linearity in $f$ in the equation in the claim, it suffices to
    consider the case where $\lcS \cap V = \set{z_1 =z_2 = \dotsm
      z_\sigma = 0}$, $\res{\sect_{(p)}}_{V} = z_{\sigma+1} \dotsm z_{\sigma_V}$ and
    \begin{equation*}
      f = dz_1 \wedge dz_2 \wedge \dotsm \wedge dz_\sigma \wedge g_p
      \sect_{(p)} 
    \end{equation*}
    (in which $g_p$ is abused to mean a $(n-\sigma,0)$-form on $V$).
    Write also
    \begin{equation*}
      \xi =: dz_1 \wedge dz_2 \wedge \dotsm \wedge dz_\sigma \wedge
      \xi_p
      \quad\text{ such that }\;\;
      \res{\xi_p}_{\lcS<V>} = \PRes[\lcS](\frac{\xi}{\sect_D})
      \cdot \sect_{(p)} =\rs*\xi_p \; .
    \end{equation*}
    Let $(r_j, \theta_j)$ be the polar coordinates of the $z_j$-plane
    and set 
    \begin{equation*}
      F_0 :=\inner{\xi_p}{g_p} \:e^{-\vphi_L}
      \quad\text{ and }\quad
      F_j :=\fdiff{r_j} \paren{\frac{F_j}{r_j^2 \fdiff{r_j^2} \psi_D}}
      \quad\text{ for } j=1,\dots, \sigma \; .
    \end{equation*}
    Notice that $\fdiff{r_j} \sect_{(p)} = 0$ and coefficients of
    $F_j$ are in $\smooth_{X\:c\,*}$ on $V$ for $j=1,\dots,\sigma$.
    It then follows from the similar computation in
    \cite{Chan&Choi_ext-with-lcv-codim-1}*{Prop.~2.2.1} or
    \cite{Chan&Choi_injectivity-I}*{Thm.~2.6.1} that
    \begin{align*}
      \eps \int_V
      \frac{\inner{\xi}{f} \:e^{-\phi_D-\vphi_L}}{\abs{\psi_D}^{\sigma
      +\eps}}
      &=\eps \int_V
        \frac{\inner{\xi_p}{g_p} \:e^{-\vphi_L}}{\sect_{(p)}\:\abs{\psi_D}^{\sigma
        +\eps}} \wedge \bigwedge_{j=1}^\sigma \frac{\pi\ibar\:dz_j
        \wedge d\conj{z_j}}{\abs{z_j}^2}
      \\
      &=\eps \int_V \frac{F_0}{\sect_{(p)}\:\abs{\psi_D}^{\sigma +\eps}}
        \prod_{j=1}^\sigma d\log r_j^2 \cdot
        \underbrace{\prod_{j=1}^\sigma \frac{d\theta_j}2}_{=:\:
        \vect{d\theta}}
      \\
      &=\frac\eps{\sigma-1+\eps}
        \int_{V} \frac{F_0}{\sect_{(p)}\:r_1^2 \fdiff{r_1^2}\psi_D}
        \:d\paren{\frac{1}{\abs{\psi_D}^{\sigma-1+\eps}}}
        \prod_{j=2}^{\sigma} d\log r_j^2 
        \cdot \vect{d\theta} \\
      &\overset{\mathclap{\text{int.~by parts}}}=
        \quad\;\;
        \frac{-\eps}{\sigma-1+\eps}
        \int_{V}
        \frac{\alert{F_1}}{\sect_{(p)}\:\abs{\psi_D}^{\sigma-1+\eps}}
        \prod_{j=2}^{\sigma} d\log r_j^2 
        \cdot dr_1 \:\vect{d\theta} \\
      &= \dotsm =
        \frac{(-1)^{\sigma} \eps} {\prod_{j=1}^{\sigma} \paren{\sigma-j+\eps}} 
        \int_{V}
        \frac{F_{\alert{\sigma}}}{\sect_{(p)}\:\abs{\psi_D}^{\eps}}
        \prod_{j=1}^{\sigma} dr_j
        \cdot \vect{d\theta} \; .
    \end{align*}
    Note that $\frac{1}{\sect_{(p)}}$ is integrable on $V$, so the
    integral on the far right-hand-side converges for all $\eps \geq
    0$. 
    Letting $\eps \tendsto 0^+$ on both sides, the desired formula
    then follows from the fundamental theorem of calculus.

    When $f$ or $\xi$ belongs to $\logKX[L] \otimes \smooth_{X\,*}
    \cdot\aidlof-1\paren{V}$, the residue formula in the proposition
    holds even when $\sigma$ is replaced by $\sigma-1$, which implies
    that the integral $\int_V \frac{\inner{\xi}{f}
      \:e^{-\phi_D-\vphi_L}}{\abs{\psi_D}^{\sigma +\eps}}$ converges
    for all $\abs\eps < 1$, hence the last claim.
  \end{proof}

  When restriction to a subspace of codimension $1$ is considered,
  there is a more classical kernel for obtaining the residue formula.
  As an illustration, the residue formula from $X$ to $\lcc|1|'$ is
  proved in the following proposition (which is applied to the case
  where the residue from $\lcc'$ to $\lcc+1'$ is considered later). 
  Recall that $\lcc|1|' =D =\sum_{i \in \Iset||} D_i$, where $D_i =
  \lcS|1|[i]$ and $\Iset|| =\Iset|1|$ are set for convenience.
  \begin{prop} \label{prop:residue-formula-classical-kernel}
    Given any admissible open set $V \subset X$ and any compactly
    supported section $f \in
    \logKX[L] \otimes \smooth_{X \:c\,*} \cdot\aidlof|1|\paren{V}$
    such that $\Res^1(f) = g =\paren{g_i}_{i\in\Iset||}$, one
    has, for any $\xi \in \logKX[L] \otimes \smooth_{X \, *}\paren{V}$,
    \begin{align*}
      \lim_{\eps \tendsto 0^+} \eps \int_V
      \inner{\xi}{f} \:e^{-\phi_D-\vphi_L} e^{-\eps\abs{\psi_D}}
      &=\sum_{i \in \Iset||} \pi
        \int_{D_i \cap V} \inner{\frac{\rs*\xi_i}{\sect_{(i)}}}{\: g_i}
        \:e^{-\vphi_L} \\
      &=\sum_{i \in \Iset||}
        \pi
        \int_{D_i \cap V} \inner{\rs*\xi_i}{\: g_i \sect_{(i)}}
        \:e^{-\phi_{(i)}-\vphi_L}
    \end{align*}
    which is finite,
    where $\phi_{(i)} :=\log\abs{\sect_{(i)}}^2$ and
    \begin{equation*}
      \rs*\xi_i := \PRes[D_i](\frac{\xi}{\sect_D}) \cdot \sect_{(i)}
      \in K_{D_i} \otimes \Diff_i D \otimes \res L_{D_i} \otimes
      \smooth_{X\:c\, *}\paren{D_i \cap V} \; .
    \end{equation*}    
  \end{prop}

  \begin{proof}
    As before, it suffices to consider the case where $D_i \cap V
    =\set{z_1 =0}$, $\res{\sect_{(i)}}_V =z_2 \dotsm z_{\sigma_V}$ and
    \begin{equation*}
      f = dz_1 \wedge g_i \sect_{(i)} \; .
    \end{equation*}
    Write also
    \begin{equation*}
      \xi =: dz_1 \wedge \xi_i
      \quad\text{ such that }\;\;
      \res{\xi_i}_{D_i \cap V} =\PRes[D_i](\frac{\xi}{\sect_D}) \cdot
      \sect_{(i)} =\rs*\xi_i \; .
    \end{equation*}
    Essentially the same computation as in Proposition
    \ref{prop:residue-product-X-to-lcS} yields 
    \begin{align*}
      \eps \int_V \inner{\xi}{f} \:e^{-\phi_D-\vphi_L}
      e^{-\eps\abs{\psi_D}}
      =&~\eps \int_V \frac{\inner{\xi_i}{g_i}
         \:e^{-\vphi_L}}{\sect_{(i)}} \wedge
         e^{-\eps\abs{\psi_D}}
         \frac{
         \pi\ibar\:dz_1 \wedge d\conj{z_1}
         }{\abs{z_1}^2}
      \\
      =&~\int_V \frac{\inner{\xi_i}{g_i}
         \:e^{-\vphi_L}}{\sect_{(i)} \:r_1^2 \fdiff{r_1^2} \psi_D} 
         \:d\paren{e^{-\eps\abs{\psi_D}}} \:
         \frac{d\theta_1}{2}
      \\
      \overset{\mathclap{\text{int.~by parts}}}=
       &~\quad\;\;
         -\int_V \fdiff{r_1} \paren{\frac{\inner{\xi_i}{g_i}
         \:e^{-\vphi_L}}{r_1^2 \fdiff{r_1^2} \psi_D} }
         \:\frac{e^{-\eps\abs{\psi_D}}}{\sect_{(i)}} \:dr_1 \:
         \frac{d\theta_1}{2}
      \\
      \mathclap{\xrightarrow{\eps \tendsto 0^+}\;\;}
       &~\quad\;
         -\int_V \fdiff{r_1} \paren{\frac{\inner{\xi_i}{g_i}
         \:e^{-\vphi_L}}{r_1^2 \fdiff{r_1^2} \psi_D} }
         \:\frac{1}{\sect_{(i)}} \:dr_1 \:
         \frac{d\theta_1}{2}
      \\
      =&~\pi \int_{\mathrlap{D_i \cap V}} \;\;\; \frac{\inner{\rs*\xi_i}{g_i}
         \:e^{-\vphi_L}}{\sect_{(i)}}
         =\pi \int_{D_i \cap V} \inner{\rs*\xi_i}{g_i \sect_{(i)}}
         \:e^{-\phi_{(i)}-\vphi_L} \; .
    \end{align*}
    Note that the convergence of the integral obtained right after
    integration by parts follows from the same reasoning as in
    Proposition \ref{prop:residue-product-X-to-lcS}.
  \end{proof}

  Proposition \ref{prop:residue-formula-classical-kernel} facilitates the
  following residue computation.

  \begin{prop} \label{prop:res-formula-dbar-exact-dot-harmonic}
    Given the decomposition $\lcc' = \bigcup_{p\in\Iset} \lcS$, let
    $u_p$ be a \emph{harmonic} $K_{\lcS} \otimes \res L_{\lcS}$-valued
    $(0,q)$-form on $\lcS$ with respect to the norm
    $\norm\cdot_{\lcS}$ for each $p \in \Iset$.
    Given also the decomposition $\lcc+1' = \bigcup_{b\in\Iset+1}
    \lcS+1[b]$, for any $\lcS$ and $\lcS+1[b]$ such that $\lcS+1[b]
    \subset \lcS$, let $\sgn{b:p}$ be the sign such that
    \begin{equation*}
      \PRes[\lcS+1[b]] =\sgn{b:p} \:\PRes[\lcS+1[b] | \lcS] \circ
      \PRes[\lcS] \; ,
    \end{equation*}
    where $\PRes[\lcS+1[b] | \lcS]$ denotes the Poincar\'e residue map
    from $\lcS$ to $\lcS+1[b]$.
    With the finite cover $\cvr V$ and partition of unity
    $\set{\rho^i}_{i \in I}$ given in Section \ref{subsec:notation},
    let $\set{\gamma_{\idx 1.q}}_{\idx 1,q \in I}$ be a
    $\logKX[L]$-valued \v Cech $(q-1)$-cochain with respect to $\cvr
    V$ and set 
    \begin{gather*}
      \rs\gamma_{p; \:\idx 1.q} :=\PRes[\lcS](\frac{\gamma_{\idx 1.q}}{\sect_D})
      \cdot \sect_{(p)} \; , \quad
      v_{p} := \sum_{\idx 1,q \in I} \underbrace{
        \dbar\rho^{i_q} \wedge \dotsm
        \wedge \dbar\rho^{i_2} \cdot \rho^{i_1}
      }_{=: \: \paren{\dbar\rho}^{\idx q.1}} \rs*\gamma_{p;\:\idx 1.q}
      \quad\text{ on } \lcS \\
      \text{and }\quad
      \rs\gamma_{b; \:\idx 1.q} :=\PRes[\lcS+1[b]](\frac{\gamma_{\idx 1.q}}{\sect_D})
      \cdot \sect_{(b)} \; , \quad
      v_{b} := \sum_{\idx 1,q \in I} \paren{\dbar\rho}^{\idx q.1}
      \rs*\gamma_{b;\:\idx 1.q}
      \quad\text{ on } \lcS+1[b] \; .
    \end{gather*}
    Then, after setting $\iinner{\cdot}{\cdot}_{\lcS, \phi_{(p)}}
    :=\iinner{\cdot}{\cdot \:e^{-\phi_{(p)}}}_{\lcS}$ (and
    similarly for $\iinner{\cdot}{\cdot}_{\lcS+1[b], \phi_{(b)}}$), one has
    \begin{equation*}
      \sum_{p\in\Iset} \iinner{\dbar v_{p}}{
        u_p\sect_{(p)}}_{\lcS,\phi_{(p)}}
      =-\sigma \smashoperator[l]{\sum_{b\in\Iset+1}} \iinner{v_{b} \:}{\quad\;
        \smashoperator{\sum_{p\in\Iset \colon \lcS+1[b] \subset
            \lcS}} \;\;
        \sgn{b:p} \: \PRes[\lcS+1[b] | \lcS](\idxup{\diff\psi_{(p)}}.
         u_p) \cdot \sect_{(b)}
      }_{\lcS+1[b], \phi_{(b)}} \; ,
    \end{equation*}
    where $\psi_{(p)} :=\phi_{(p)} -\sm\vphi_{(p)}$ and
    $\sm\vphi_{(p)}$ is some smooth potential on $\Diff_p D$.
  \end{prop}

  \begin{proof}
    Notice that $v_{p}$ is smooth on $\lcS$ but not necessarily
    locally $L^2$ with respect to the weight $e^{-\phi_{(p)}}$.
    An integration by parts is done via the use of Proposition
    \ref{prop:residue-formula-classical-kernel}, which yields 
    \begin{align*}
      &~\sum_{p\in \Iset} \iinner{\dbar v_{p}}{ u_p
        \sect_{(p)}}_{\lcS, \phi_{(p)}}
      \\
      \xleftarrow{\eps \tendsto 0^+}
      &~\sum_{p \in \Iset} \iinner{
        e^{-\eps \abs{\psi_{(p)}}} \:\dbar v_{p}
        }{ u_p \sect_{(p)}}_{\lcS, \phi_{(p)}}
      \\
      =&~\sum_{p \in \Iset} \paren{
         \cancelto{0 \;\;\;(\because~u_p \text{ harmonic, Lemma \ref{lem:su-harmonicity}})}{\iinner{
         \dbar\paren{e^{-\eps \abs{\psi_{(p)}}} \: v_{p}}
         }{ u_p \sect_{(p)}}_{\mathrlap{\lcS, \phi_{(p)}}}}
         \quad\;\; - \eps 
         \iinner{
         e^{-\eps \abs{\psi_{(p)}}} \:v_{p}
         }{\:\idxup{\diff\psi_{(p)}}.  u_p \sect_{(p)}}_{\lcS,
         \phi_{(p)}}
         }
      \\
      =&~-\sum_{p \in \Iset} \sum_{\idx 1,q \in I} \eps \:
         \iinner{
         e^{-\eps \abs{\psi_{(p)}}} \: 
         \rs*\gamma_{p;\:\idx 1.q}
         }{\:
         \idxup{\diff\rho},[\idx 1.q] .
         \paren{\idxup{\diff\psi_{(p)}}.  u_p \sect_{(p)}}
         }_{\lcS, \phi_{(p)}}
      \\
      \xrightarrow[\text{Prop.~\ref{prop:residue-formula-classical-kernel}}]{\eps
      \tendsto 0^+} 
      &~-\smashoperator[l]{\sum_{\idx 1,q \in I}} \sum_{p \in \Iset}
        \sum_{k=\sigma +1}^{\mathclap{\sigma_{V_{\idx 1.q}}}} \sigma
        \iinner{
        \PRes[p(k)](
        \frac{\rs*\gamma_{p;\:\idx 1.q}}{\sect_{(p)}}
        )
        }{\:
        \idxup{\diff\rho},[\idx 1.q] .
        \PRes[p(k)](\idxup{\diff\psi_{(p)}}.  u_p)
        }_{\lcS \cap \set{z_{p(k)} =0}}
        \; ,
    \end{align*}
    where $\idxup{\diff\rho},[\idx 1.q] . \cdot$ is the adjoint
    of $\paren{\dbar\rho}^{\idx q.1} \cdot$, and $\PRes[p(k)]$ denotes
    the Poincar\'e residue map from $\lcS$ to $\lcS \cap \set{z_{p(k)}=0}$. 
    The last limit is justified as follows.
    On the admissible open set $V_{\idx 1.q}$, consider a holomorphic
    coordinate system $(z_1, \dots, z_n)$ such that $\lcS \cap V_{\idx
    1.q}
    =\set{z_{p(1)} = \dotsm =z_{p(\sigma)} =0}$ and
    $\sect_{(p)} =z_{p(\sigma+1)} \dotsm z_{p(\sigma_V)}$ (write
    $\sigma_{V}$ for $\sigma_{V_{\idx 1.q}}$ for convenience).
    Note that
    \begin{equation*}
      \diff\psi_{(p)} =\sum_{k =\sigma +1}^{\sigma_V}
      \frac{dz_{p(k)}}{z_{p(k)}} -\diff\sm\vphi_{(p)} \quad\text{ on }
      V_{\idx 1.q} \; .
    \end{equation*}
    It follows that, on $\lcS \cap V_{\idx 1.q}$,
    \begin{equation*}
      \text{coef.~of }\:
      \idxup{\diff\rho},[\idx 1.q].
      \paren{\idxup{\diff\psi_{(p)}}.  u_p \sect_{(p)}}
      \in
      \smooth_{\lcS \:c} \cdot\res{\defidlof{\lcc+2'}}_{\lcS}
      \begin{aligned}[t]
        &=\smooth_{\lcS \:c} \cdot\mtidlof<\lcS>{\vphi_L} \cdot
        \res{\defidlof{\lcc+2'}}_{\lcS} \;\;\footnotemark
        \\
        &=\smooth_{\lcS \:c} \cdot\aidlof|1|<\lcS>{\vphi_L}[\psi_{(p)}]
      \end{aligned}
    \end{equation*}%
    \footnotetext{
      Recall that $\defidlof{\lcc+2'}$ is generated on $X$ by
      $\sect_{(b)}$ treated as local
      functions for all $b \in \Iset+1$.
      On an admissible open set $V$, one has $\defidlof{\lcc+2'}
      =\genbyd{z_{b(\sigma+2)} \dotsm
        z_{b(\sigma_V)}}{b \in \Iset+1 \text{ such that } \lcS+1[b] \cap
        V \neq \emptyset}$.
    }%
    and, therefore, one can apply Proposition
    \ref{prop:residue-formula-classical-kernel} (with $\lcS$ in place
    of $X$, $\psi_{(p)}$ in place of $\psi_D$) to each inner product
    $\eps \iinner{e^{-\eps \abs{\psi_{(p)}}} \dotsm}{\: \dotsm
      \idxup{\diff\psi_{(p)}} . \dotsm \sect_{(p)}}_{\lcS,\phi_{(p)}}$.
    Note also that the factor $\sigma$ comes from the normalisation of
    the norm on each lc center ($\norm\cdot_{\lcS}^2
    =\frac{\pi^\sigma}{(\sigma -1)!} \int_{\lcS} \dotsm$ and
    $\norm\cdot_{\lcS+1[b]}^2 =\frac{\pi^{\sigma+1}}{\sigma!}
    \int_{\lcS+1[b]} \dotsm$).

    On each admissible open set $V_{\idx 1.q}$, the intersection $\lcS
    \cap \set{z_{p(k)} = 0}$ is a $(\sigma+1)$-lc center $\lcS+1[b_{p,k}]
    \cap V_{\idx 1.q}$ ($\neq \emptyset$), uniquely determined by the
    choices of $p\in \Iset$ (such that $\lcS \cap V_{\idx 1.q} \neq
    \emptyset$, so $\binom{\sigma_V}{\sigma}$ choices) and $k
    =\sigma+1, \dots, \sigma_V$ (so $\sigma_V-\sigma$ choices).
    To get an indexing in terms of $b \in \Iset+1$ (such that
    $\lcS+1[b] \cap V_{\idx 1.q} \neq \emptyset$, so
    $\binom{\sigma_V}{\sigma +1}$ choices), note that each $\lcS+1[b]
    \cap V_{\idx 1.q}$ is contained in $\sigma +1$ distinct
    $\sigma$-lc centers $\lcS[p_{b,j}]$ for $j=1,\dots,\sigma+1$
    (apparently, $\sigma +1$ choices) such that
    \begin{equation*}
      \lcS+1[b] \cap V_{\idx 1.q} = \lcS[p_{b,j}] \cap \set{z_{b(j)} = 0} \; .
    \end{equation*}
    (One can verify $\sum_{p \in \Iset} \sum_{k=\sigma
      +1}^{\sigma_{V}} \dotsm = \sum_{b \in
      \Iset+1} \sum_{j=1}^{\sigma +1} \dotsm$ by first noting that
    $\binom{\sigma_V}{\sigma} (\sigma_V -\sigma)
    =\binom{\sigma_V}{\sigma +1} (\sigma+1)$.)
    With such choice of indexing, one has
    \begin{equation*}
      \frac{\rs*\gamma_{b;\: \idx 1.q}}{\sect_{(b)}}
      :=\PRes[\lcS+1[b]](\frac{\gamma_{\idx 1.q}}{\sect_D})
      =\sgn{b:p_{b,j}} \:
      \PRes[b(j)](\frac{\rs*\gamma_{p_{b,j};\:\idx
          1.q}}{\sect_{(p_{b,j})}})
    \end{equation*}
    (noticing that 
    $\sect_{(p_{b,j})} = z_{b(j)} \sect_{(b)}$).
    As a result, the expression in question becomes
    \begin{align*}
      &-\smashoperator[l]{\sum_{\idx 1,q \in I}} \sum_{b \in \Iset+1}
        \sum_{j=1}^{\sigma +1} \sigma
        \iinner{ \sgn{b:p_{b,j}}\:
        \frac{\rs*\gamma_{b;\:\idx 1.q}}{\sect_{(b)}}
        }{\: 
        \idxup{\diff\rho},[\idx 1.q] .
        \PRes[b(j)](\idxup{\diff\psi_{(p_{b,j})}} . u_{p_{b,j}})
        }_{\lcS+1[b]}
      \\
      =&-\smashoperator[l]{\sum_{\idx 1,q \in I}}
        \sum_{b \in \Iset+1}
        \sigma
        \iinner{
        \paren{\dbar\rho}^{\idx q.1} \rs*\gamma_{b;\:\idx 1.q}
        \:}{  \sum_{j=1}^{\sigma +1} \sgn{b:p_{b,j}}\:
        \PRes[b(j)](\idxup{\diff\psi_{(p_{b,j})}} . u_{p_{b,j}})
        \cdot \sect_{(b)}
        }_{\lcS+1[b], \phi_{(b)}}
      \\
      =&-\sigma \sum_{b \in \Iset+1} \iinner{
        v_b
        \:}{\quad\;
        \smashoperator{\sum_{p\in\Iset \colon \lcS+1[b] \subset
        \lcS}} \;\;
        \sgn{b:p}\:
        \PRes[\lcS+1[b] | \lcS](\idxup{\diff\psi_{(p)}}.  u_{p})
        \cdot \sect_{(b)}
        }_{\lcS+1[b], \phi_{(b)}} \; . \qedhere
    \end{align*}
  \end{proof}

}

\subsection{Restriction of harmonic differential forms to hypersurfaces}\label{subsec:harmonic}

{
  
  Let $(X,\omega)$ be a K\"ahler manifold equipped with a holomorphic
  line bundle $L$ equipped with a smooth potential $\varphi_L$ such
  that $\ibddbar\varphi_L\ge0$ and let $D$ be an snc divisor in $X$
  written as 
  \begin{equation*}
    D=\sum_{p\in I_D}D_p \; ,
  \end{equation*}
  where $D_p$ is an irreducible component for $p\in I_D$.
  We define the map $\HRes_p \colon \mathscr
  A_X^{0,q}(X,K_X\otimes L)\rightarrow \mathscr
  A_{D_p}^{0,q-1}(D_p,K_{D_p}\otimes L\vert_{D_p})$ by 
  \begin{equation*}
    \HRes_p(u)
    =
    \PRes[D_p](\idxup{\partial\psi_D} . u) \;\;\;\text{for}\;\;\;p\in
    I_D \; ,
  \end{equation*}
  where $\mathcal{R}_{D_p}$ is the Poincar\'e residue map (see, for
  example, \cite{Griffiths&Harris}*{p.147} or \cite{Kollar_Sing-of-MMP}*{\S 4.18}). 
  Notice that, as in \cite{Chan&Choi_injectivity-I}*{\textsection2.6},
  the map $\mathcal R_{D_p}$ is extended to send sections of
  $K_X\otimes\overline{\bold\Omega}_X^q$ to those of
  $K_{D_p}\otimes\overline{\bold\Omega}_X^q\vert_{D_p}$. 
  Let $(U;z^1,\dots,z^n)$ be a local holomorphic coordinate system
  around $x\in D_p\subset X$ satisfying 
  \begin{enumerate}[label=(\roman*), ref=\roman*]
  \item  \label{item:admissible-open-U} 
    $D_p\cap U=\{z\in U:z^1=0\}$ and 
    $D\cap U=\set{z^1\cdots z^{\sigma_U}=0}$,
    and
  \item  \label{item:psi_D-in-admissible-open-U} 
    $\psi_D=\sum_{j=1}^{\sigma_U}
    \log\abs{z^j}^2-\sm\varphi_D$ on $U$.
  \end{enumerate}
  Since $\del\psi_D=\sum_{j=1}^{\sigma_U}\frac{dz^j}{z^j}-\del\sm\varphi_D$, it follows that
  \begin{equation*}
    \HRes_p(u)
    =
    \mathcal{R}_{D_p}\left(\idxup{\partial\psi_D}.u\right)
    =
    \mathcal{R}_{D_p}\paren{\idxup{\frac{dz^1}{z^1}}.u}
    =
    \paren{\idxup{dz^1}.\widetilde u_p}\big\vert_{D_p} \; ,
  \end{equation*}
  where $\rs u_p := \fdiff{z^1} \ctrt u$ (so $u=dz^1\wedge\widetilde{u}_p$).
  In particular, $\HRes_p$ does not depend on the
  choice of the Hermitian metric $\sm\varphi_{D}$. 
  It follows from the above formula that
  $\HRes_p(u)$ is actually a $K_{D_p}\otimes
  L\vert_{D_p}$-valued $(0,q-1)$-form on $D_p$ (not only a
  $\overline{\bold\Omega}_X^{q-1}\vert_{D_p}$-valued section).

  First we notice that $\dfadj$-closedness is preserved by
  $\HRes_p$ on a K\"ahler manifold.
  \begin{prop} \label{prop:harmonic-residue}
    If $u$ is a $\dfadj$-closed $K_X\otimes L$-valued $(0,q)$-form on
    $X$, then $\HRes_p(u)$ is a $\dfadj$-closed
    $K_{D_p}\otimes L\vert_{D_p}$-valued $(0,q-1)$-form on $D_p$. 
  \end{prop}

  \begin{proof}
    It is enough to show that it vanishes at the given point $x\in D_p$.
    Let $(z^1,\dots,z^n)$ be a local holomorphic coordinate system around
    $x$ in $X$ satisfying \eqref{item:admissible-open-U} 
    and \eqref{item:psi_D-in-admissible-open-U}. 
    Since $(D_p,\omega\vert_{D_p})$ is a smooth $(n-1)$-dimensional
    K\"ahler manifold, by a linear change of coordinates
    $(z^1,\ldots,z^n)$ and a quadratic change of coordinates in
    $(z^2,\dots,z^n)$ of $D_p$ we may assume that
    \begin{enumerate}[resume*]
    \item 
      $g_{i\bar j}(x)=I_n$ where $I_n$ is the $n\times n$ identity matrix.
    \item 
      $dg_{\alpha\bar\beta}(x)=0$ for $2\le\alpha\le n$ and $2\le\beta\le n$.
    \end{enumerate}
    Since $\displaystyle\idxup{dz^1}=g^{\bar
      j1}\pd{}{\overline{z^j}}$ (under Einstein summation convention),
    we have
    \begin{equation}\label{E:local_expression}
      dz^1 \wedge \paren{\idxup{dz^1} . \widetilde u_p}_{\bar j_1,\ldots,\bar j_{q-1}}
      =
      g^{\bar j1}
      u_{\bar j\bar j_1,\ldots,\bar j_{q-1}} \; .
    \end{equation}
    For the sake of convenience, let the Latin indices $i,j,k,...$ run
    from $1$ to $n$ and let the Greek indices
    $\alpha,\beta,\gamma,...$ run from $2$ to $n$ in this proof. 
    Let $\varphi_{K_X}$ and $\varphi_{K_{D_p}}$ be respectively the
    potentials on $K_X$ and $K_{D_p}$ induced by the K\"ahler metric
    $\omega$, which are written as 
    \begin{equation*}
      \varphi_{K_X}
      =
      \log\det\paren{g_{i\bar j}}_{1\le i,j\le n}
      \;\;\;\text{and}\;\;\;
      \varphi_{K_{D_p}}
      =
      \log\det\paren{g_{\alpha\bar\beta}}_{2\le\alpha,\beta\le n} \; .
    \end{equation*}
    This yields, at the given point $x\in D_p$,
    \begin{align*}
      \del_\gamma\varphi_{K_X}
      &=
	\del_\gamma\log\det g
	=
	\Tr\paren{\del_\gamma g\cdot g^{-1}}
      \\
      &=
	\sum_{i,j=1}^n\pd{g_{i\bar j}}{z^\gamma} g^{\bar j i}
	\;\;\overset{\mathclap{(\text{at } x)}}=\;\;
	g^{\bar11}\del_\gamma g_{1\bar1}
	+
	\sum_{\alpha,\beta=2}^n\pd{g_{\alpha\bar\beta}}{z^\gamma}g^{\alpha\bar\beta}
      \\
      &=
	g^{\conj11}\del_\gamma g_{1\conj1}
	+
	\del_\gamma\log\det\paren{g\vert_{D_p}}
	=
	g^{\conj11}\del_\gamma g_{1\conj1}
	+
	\del_\gamma\varphi_{K_{D_p}}
	\;\;\overset{\mathclap{(\text{at } x)}}=\;\;
	g^{\conj11}\del_\gamma g_{1\conj1} \; .
    \end{align*}
    \newcommand{\bbeta}{{\boldsymbol{\beta}}}%
    \newcommand{\KDp}{{\smash[b]{K_{D_p}}}}%
    Choose a local frame of $L$ in a neighbourhood of $x$ in $X$
    such that
    \begin{equation*}
      \diff\vphi_L(x) = 0  \; .
    \end{equation*}
    It follows from \eqref{E:local_expression} and the definition of
    $\dfadj$ on $D_p$ (see, for example,
    \cite{Siu}*{(1.3.2)}) that, for any 
    multi-indices ${\bbeta}_{q-2} =(\idx[\beta]1,{q-2})$ and at the
    given point $x \in D_p$,
    \begin{align*}
      dz^1 \wedge \paren{\dfadj \HRes_p(u)}_{\conj\bbeta_{q-2}}
      &=
        -g^{\conj\beta\gamma}
	\nabla_\gamma 
	\paren{
        g^{\alert{\conj j} 1}
        u_{\alert{\conj j} \conj\beta\ov{\boldsymbol\beta}_{q-2}}
	}
      \\
      &=-g^{\conj\beta \gamma}\diff_\gamma \paren{g^{\alert{\conj j} 1}
        u_{\alert{\conj j} \conj\beta \conj\bbeta_{q-2}}}
        +g^{\conj\beta \gamma} \smash[t]{\cancelto{0}{\paren{
        \diff_\gamma \vphi_\KDp +\diff_\gamma \vphi_L
        }}}
        \cdot g^{\alert{\conj j} 1} u_{\alert{\conj j} \conj\beta \conj\bbeta_{q-2}}
      \\
      &=-g^{\conj\beta \gamma} g^{\alert{\conj j} 1} \diff_\gamma
        u_{\alert{\conj j} \conj\beta \conj\bbeta_{q-2}}
        -g^{\conj\beta \gamma} \diff_\gamma g^{\alert{\conj j} 1}
        \cdot u_{\alert{\conj j} \conj\beta \conj\bbeta_{q-2}}
      \\
      &=-g^{\conj\beta \gamma} g^{\conj 1 1}
        \diff_\gamma u_{\conj 1 \conj\beta \conj\bbeta_{q-2}}
        +g^{\conj\beta \gamma} g^{\alert{\conj j k}} \diff_\gamma
        g_{\alert k \conj 1} \cdot g^{\conj 1 1} u_{\alert{\conj j}
        \conj\beta \conj\bbeta_{q-2}}
      \\
      &=g^{\conj 1 1} \paren{
        -g^{\conj\beta \gamma} \diff_\gamma
        u_{\conj 1 \conj\beta \conj\bbeta_{q-2}}
        +g^{\conj\beta\gamma} g^{\conj 1 1} \diff_\gamma
        g_{1\conj 1} \cdot
        u_{\conj 1 \conj\beta \conj\bbeta_{q-2}}
        +g^{\conj\beta\gamma} g^{\alert{\conj j} \alpha}
        \diff_\gamma g_{\alpha \conj 1} \cdot
        u_{\alert{\conj j} \conj\beta \conj\bbeta_{q-2}}
        }
      \\
      &=g^{\conj 1 1} \paren{
        g^{\alert{\conj k j}} \diff_{\alert{j}}
        u_{\alert{\conj k} \conj 1 \conj\bbeta_{q-2}}
        -g^{\alert{\conj k j}} \diff_{\alert{j}}\vphi_{K_X} \cdot
        u_{\alert{\conj k} \conj 1 \conj\bbeta_{q-2}}
        }
        +g^{\conj 1 1} g^{\conj \gamma \gamma} g^{\conj\alpha
        \alpha} \diff_{\gamma} g_{\alpha \conj 1} \cdot
        u_{\conj\alpha \conj\gamma \conj\bbeta_{q-2}}
      \\
      &=-g^{\conj 1 1} \paren{\dfadj u}_{\conj 1 \conj\bbeta_{q-2}}
        +g^{\conj 1 1} g^{\conj \gamma \gamma} g^{\conj\alpha
        \alpha} \diff_{\gamma} g_{\alpha \conj 1} \cdot
        u_{\conj\alpha \conj\gamma \conj\bbeta_{q-2}} \; .
    \end{align*}
    Since $\del_\gamma g_{\alpha\conj1}$ is symmetric in $\alpha,
    \gamma$ (for $X$ being K\"ahler) while $u_{\conj\alpha\conj\gamma
      \conj{\bbeta}_{q-2}}$ is anti-symmetric in $\alpha, \gamma$, the
    last term in the expression above vanishes.
    As $\dfadj u = 0$ on $X$ by assumption, the proof is thus
    completed after applying $\fdiff{z^1} \ctrt$ to both sides. \qedhere
  \end{proof}

  Furthermore, we claim that, if
  $u$ satisfies $\dbar u = 0$ and $\nabla^{(0,1)}u = 0$, then
  $\HRes_p(u)$ is $\dbar$-closed.
  This is shown via the following formula, which is a special case and
  a slight variant of \cite{Donnelly&Xavier}*{(2.4)} and
  \cite{Ohsawa&Takegoshi-spectral_seq}*{Prop.~1.5} (see also
  \cite{Takegoshi_higher-direct-images}*{(1.9)} and
  \cite{Matsumura_injectivity-Kaehler}*{Lemma 2.1}).

  \newcommand{\lcSb}{\lcS+1[b]}
  \newcommand{\idxj}{\idx[\conj j]}
  
  \begin{lemma}[cf.~\cite{Donnelly&Xavier}*{(2.4)},
    \cite{Ohsawa&Takegoshi-spectral_seq}*{Prop.~1.5},
    \cite{Takegoshi_higher-direct-images}*{(1.9)} and
    \cite{Matsumura_injectivity-Kaehler}*{Lemma
      2.1}] \label{lem:commutator-dbar-ctrt}
    Let $\vphi$ be a smooth function and $u$ be a smooth
    ($K_X$-valued) $(0,q)$-form on a K\"ahler manifold.
    They satisfy the formula
    \begin{equation*}
      \dbar\paren{\idxup{\diff\vphi}.  u}
      =\idxup{\ibddbar\vphi} . u
      -\idxup{\diff\vphi} . \paren{\dbar u}
      +\idxup{\diff\vphi} \cdot \nabla^{(0,1)}_\bullet u \; ,
    \end{equation*}%
    or, when a local holomorphic coordinate system is fixed and
    the Einstein summation convention is applied, 
    \begin{equation*}
      \paren{\dbar\paren{\idxup{\diff\vphi} . u}}_{\conj J_{q}}
      =\sum_{\nu=1}^q \diff^{\conj\ell} \diff_{\conj j_\nu} \vphi \:
      u_{\idxj 1[\dotsm (\conj \ell)_\nu].q}
      -\diff^{\conj\ell}\vphi  \:\paren{\dbar u}_{\conj\ell\conj J_q}
      +\diff^{\conj\ell}\vphi \:\nabla_{\conj\ell} u_{\conj J_q} 
    \end{equation*}
    for any multi-indices $J_q = (\idx[j]1,q)$, pointwisely.
  \end{lemma}

  \begin{proof}
    A direct computation yields
    \begin{align*}
      \paren{\dbar\paren{\idxup{\diff\vphi} . u}}_{\conj
      J_{q}}
      &=\sum_{\nu=1}^q (-1)^{\nu-1} \diff_{\conj j_\nu}
        \paren{\idxup{\diff\vphi}.  u}_{\idxj 1[\dotsm \widehat
        {\conj j}_\nu].q}
        =\sum_{\nu=1}^q (-1)^{\nu-1} \diff_{\conj j_\nu}
        \paren{\diff_{\ell}\vphi \: u^\ell_{\;\idxj 1[\dotsm
        \widehat{\conj j}_\nu].q}}
      \\
      &=\sum_{\nu=1}^q (-1)^{\nu-1} \paren{
        \diff_{\conj j_\nu}\diff_{\ell}\vphi \: u^{\ell}_{\;\idxj 1[\dotsm
        \widehat {\conj j}_\nu].q}
        +\diff_{\ell}\vphi \: \nabla_{\conj j_\nu} u^\ell_{\;\idxj 1[\dotsm
        \widehat {\conj j}_\nu].q}
        }
      \\
      &=\sum_{\nu=1}^q
        \diff^{\conj \ell}\diff_{\conj j_\nu}\vphi \: u_{\idxj 1[\dotsm
        (\conj\ell)_\nu].q}
        -\diff^{\conj\ell}\vphi \sum_{\nu=1}^q (-1)^{\nu} 
        \nabla_{\conj j_\nu} u_{\conj\ell \idxj 1[\dotsm
        \widehat {\conj j}_\nu].q}
        \begin{aligned}[t]
          &-\diff^{\conj\ell}\vphi \: \nabla_{\conj \ell} u_{\conj
            J_q} \\
          &+\diff^{\conj\ell}\vphi \: \nabla_{\conj \ell}
          u_{\conj J_q}
        \end{aligned}
      \\
      &=\sum_{\nu=1}^q
        \diff^{\conj \ell}\diff_{\conj j_\nu}\vphi \: u_{\idxj 1[\dotsm
        (\conj\ell)_\nu].q}
        -\diff^{\conj\ell}\vphi
        \:\paren{\dbar u}_{\conj\ell\conj J_q}
        +\diff^{\conj\ell}\vphi \: \nabla_{\conj \ell} u_{\conj
        J_q} \; . \qedhere
    \end{align*}
  \end{proof}

  We see that $\HRes_p(u)$ is $\dbar$-closed by
  putting $z^1$ in place of $\vphi$ in Lemma \ref{lem:commutator-dbar-ctrt}.
  The following theorem is then immediate.
  \begin{thm} \label{thm:residue-harmonic}
    If $u$ is a harmonic $K_X\otimes L$-valued $(0,q)$-form on $X$ with
    respect to $\vphi_L$ and $\omega$ such that $\nabla^{(0,1)}u=0$,
    then $\HRes_p(u)$ is a harmonic $K_{D_p}\otimes
    L\vert_{D_p}$-valued $(0,q-1)$-form on $D_p$ with respect to
    $\varphi\vert_{D_p}$ and $\omega\vert_{D_p}$.
  \end{thm}

  \begin{proof}
    From the above discussion, $\HRes_p(u)$ is
    $\dbar$- and $\dfadj$-closed on $D_p$.
    Since $\varphi_L$ is smooth and $D_p$ is compact, it follows that
    $\HRes_p(u)\in\Dom\dbar^*$ with
    respect to $\varphi_L\vert_{D_p}$ and $\omega\vert_{D_p}$.
    This completes the proof.
  \end{proof}

  \begin{remark}
    When $\varphi_L$ has the singularity property described in
    \cite{Chan&Choi_injectivity-I}*{\S 2.2 item (2)} for $\varphi_F$,
    i.e.~$\varphi_L$ has only neat analytic singularities such that
    $P_L:=\varphi_L^{-1}(-\infty)$ is a divisor with $P_L+D$ having
    snc and that $P_L$ contains no components of $D$, the claim that
    $\HRes_p(u)\in\Dom\dbar^*$ with
    respect to $\norm\cdot_{D_p}:=\norm\cdot_{D_p,\varphi_L,\omega}$
    still holds true (under the assumption that $\omega\vert_{D_p}$ is
    a complete K\"ahler form on $D_p\setminus P_L$).
    Indeed, $\HRes_p(u)$ can be shown to be $L^2$
    with respect to $\norm\cdot_{D_p}$ by the arguments in
    \cite{Chan&Choi_injectivity-I}*{Prop.~3.2.3, Remark 3.2.4 and
      Prop.~3.3.2} (with $u$ here in place of $\frac{\rs u}{\sect_D}$
    there).
    With $\dfadj$ (with respect to $\varphi_L\vert_{D_p}$ and
    $\omega\vert_{D_p}$) being a smooth operator on $D_p\setminus P_L$
    and $\omega\vert_{D_p}$ being complete,
    $\HRes_p(u)\in\Dom\dbadj$ follows
    from the classical arguments.
  \end{remark}

}

\section{Proofs of the main results}\label{sec:proof}

\subsection{Proof of Corollary \ref{cor:main}}\label{subsec:n2}

Corollary \ref{cor:main} can be proved by adapting the 
proof of Theorem \ref{thm:main} (or, more precisely, Theorem
\ref{thm:ker-nu=ker-tau}; see Remark \ref{rem:general-commut-diagram}
for details).
The proof involves an inductive reduction of the setup to
subvarieties on which the relevant injectivity result is known
or can be proved via Enoki's arguments (i.e.~harmonic theory for
cohomology is valid).
To get an essence of the argument, here 
we deduce Corollary \ref{cor:main} from Theorem \ref{thm:main} 
using the previous work \cite{Matsumura_injectivity-lc}*{Thm.~1.6} (or
\cite{Chan&Choi_injectivity-I}*{Thm.~1.2.1}).

\begin{proof}[Proof of Corollary \ref{cor:main}]
Consider the following commutative diagram  induced by 
the short exact sequence $0 \to K_{X} \to K_{X}\otimes D \to K_{D} \to
0 $ and the multiplication map $\otimes s$: 
\begin{equation*}
  \xymatrix@R=3ex{
    \ar[d] & \ar[d]\\
    {H^q(X,  K_{X}\otimes F )} \ar[d]^-{}\ar[r] ^-{\otimes s}
    \ar[d]^-{\otimes \sect_D}
    &{H^q(X, K_{X}  \otimes F^{ \otimes{(m+1)}} )} \ar[d]\\ 
    {H^q(X,  K_{X}\otimes D \otimes F)}
    \ar[d]^-{}\ar[r] ^-{\otimes s}
    &{H^q(X,  K_{X} \otimes D \otimes F^{\otimes{(m+1)}})} \ar[d]\\ 
    {H^q(D, K_{D}\otimes F )}
    \ar[d]\ar[r]^-{\otimes s|_{D} } 
    & {H^q(D,  K_{D}\otimes F^{\otimes(m+1)} ) \; .}  \ar[d]\\ 
    & 
  }
\end{equation*}
The line bundle $M:=F^{\otimes m}$ with the metric
$h_{M}:=h_{F}^{\otimes m} = e^{-m\vphi_F}$ satisfies the curvature
assumption in Theorem \ref{thm:main} and the zero locus $s^{-1}(0)$
contains no lc centers of $(X, D)$ by assumption.
Hence, by Theorem \ref{thm:main}, the lowest multiplication map $\otimes s|_{D}$ in the diagram is injective for every $q$. 
This implies that a cohomology class $\alpha \in H^q(X,  K_{X}\otimes D \otimes F)$ with 
$s  \alpha =0 \in H^q(X,  K_{X}\otimes D \otimes F^{\otimes (m+1)})$ 
lies in the image of the vertical multiplication map $\otimes \sect_D$ on the left, 
where $\sect_D$ is the canonical section of the effective divisor $D$. 
Then, the conclusion $\alpha=0$ follows from \cite{Matsumura_injectivity-lc}*{Thm.~1.6} 
(or \cite{Chan&Choi_injectivity-I}*{Thm.~1.2.1}). 
\end{proof}

\subsection{Proof of Theorem \ref{thm:main} for a simple case} 
\label{sec:proof-of-simple-case}

In this subsection, we prove Theorem \ref{thm:main} in the simple case 
where \emph{$D$ has two components (i.e.~$D=D_{1}+D_{2}$) whose intersection
has only one irreducible component} and the degree of cohomology
groups is $q=1$.
While this case is contained in the proof presented in Section
\ref{subsec:general}, a detailed proof of it is presented here in
order to illustrate the essence of the proof in the general case
without being obscured by the notation.
The proof in Section \ref{subsec:general} follows the same arguments
but on the lower-dimensional lc strata with more components.

\begin{proof}[Proof of Theorem \ref{thm:main} in the case of $D=D_{1}+D_{2}$ and $q=1$] 

Under the given assumptions and for a given cohomology class $\alpha
\in H^1(D,  K_{D} \otimes F)$, we prove here that $\alpha $ is actually $0$
when $s  \alpha =0 \in H^1(D,  K_{D} \otimes F \otimes M)$. 

\begin{step}[``Harmonic representative'' of $\alpha$] \label{step:harmonic-rep}
We intend to work with an \emph{``optimal''} representative of $\alpha$ via the Dolbeault isomorphism, 
in the analogy of a harmonic form being the element with the minimal
$L^2$ norm in the corresponding cohomology class.
Nevertheless, at the time of writing, there is not yet a well
established theory of Dolbeault isomorphism and harmonic theory for
cohomology groups on the singular space $D$.
For this purpose, we consider the following diagram
\begin{equation}\label{h}
  \begin{aligned}
    \xymatrix@C=3.5em@R=3.5ex{
      \ar[d] & \ar[d]\\
      {\smash{\bigoplus_{p=1}^{2}} H^1(D_{p}, K_{D_{p}}\otimes F )}
      \ar[d]^-{}\ar[r] ^-{\otimes (s|_{D_{1}}, s|_{D_{2}})}
      \ar[d]^-{\tau}
      &{\smash{\bigoplus_{p=1}^{2}} H^1(D_{p},
        K_{D_{p}} \otimes F \otimes M )}
      \ar[d]\\
      {H^1(D, K_{D} \otimes F)} \ar[d]^-{}\ar[r] ^-{\otimes s}
      &{H^1(D,  K_{D} \otimes F\otimes M)} \ar[d]\\
      {H^1(D_{1}\cap D_{2}, K_{D_{1}\cap D_{2}}\otimes F )}
      \ar[d]\ar[r]^-{\otimes s|_{D_{1}\cap D_{2}} }
      &{H^1(D_{1}\cap D_{2},  K_{D_{1}\cap D_{2}}\otimes F \otimes M)} \ar[d]\\
      & }
  \end{aligned}
\end{equation}
induced from $0 \to K_{D_{1}} \oplus K_{D_{2}} \to K_{D} \to K_{D_{1} \cap D_{2}} \to 0$, 
which in turn can be obtained by tensoring $K_X \otimes D$ to the
short exact sequence of adjoint ideal sheaves
\begin{equation*}
  \renewcommand{\objectstyle}{\displaystyle}
  \xymatrix@R=3.5ex{
    {0} \ar[r]
    &{\faidlof|1|/|0|*} \ar[r] \ar[d]_-{\Res^1}^-{\isom}
    &{\faidlof|2|/|0|*} \ar[r]
    &{\faidlof|2|/|1|*} \ar[r] \ar[d]^-{\Res^2}_-{\isom}
    &{0}
    \\
    &{\residlof|1|*} 
    &&{\residlof|2|*} 
  }
\end{equation*}
where $\aidlof* :=\aidlof$ and $\residlof* := \residlof$ and
the isomorphism $\faidlof/-1* \xrightarrow[\isom]{\Res^\sigma}
\residlof*$ is induced from the residue short exact sequence in
Section \ref{subsec:residue}.
Notice that the Dolbeault isomorphism and harmonic theory are
valid on $D_1$, $D_2$ and $D_1 \cap D_2$.
The multiplication map $\otimes s|_{D_{1}\cap D_{2}}$ on the bottom
row is non-zero by the assumption on $s^{-1}(0)$ and the curvature
assumption is still satisfied after restricting $F$ and $M$ to
$D_{1}\cap D_{2}$.
Hence, Enoki's injectivity theorem can be invoked to assert that
$\otimes s|_{D_{1}\cap D_{2}}$ is injective. 
Then, by an easy diagram chasing, we can find harmonic forms $u_{p}$ for $p=1,2$ such that 
\begin{equation*}
  u_{p} \in \mathcal{H}^{n-1,1}(D_{p}; F)_{\vphi_F} \cong H^1(D_{p},  K_{D_{p}}\otimes F ) 
  \text{ with } \alpha = \tau(\eqcls{u_{1}}, \eqcls{u_{2}}), 
\end{equation*}
where $\mathcal{H}^{n-1,1}(D_{p}; F)_{\vphi_F}$ denotes 
the space of $F|_{D_{p}}$-valued harmonic forms of $(n-1,1)$-type with respect to $\res{e^{-\vphi_F}}_{D_{p}}$. 
Note that there is freedom in the choice of $(u_{1}, u_{2})$  
since $\tau$ may not be injective. 
To obtain the unique ``optimal'' representative of $\alpha$, we choose
the pair $(u_{1}, u_{2})$ 
with $\alpha = \tau(\eqcls{u_{1}}, \eqcls{u_{2}})$ that satisfies
\begin{equation}\label{eq-orth}
  (u_{1}, u_{2}) \in (\Ker \tau)^{\perp} \subset 
  \Ker \tau \oplus (\Ker \tau)^{\perp}= \bigoplus_{p=1}^{2}
  \mathcal{H}^{n-1,1}(D_{p}; F)_{\vphi_F} \; ,
\end{equation}
in which $\paren{\ker\tau}^\perp$ is the orthogonal complement of
$\ker\tau$ with respect to the (squared) residue norm
$\norm\cdot_{\lcc|1|'}^2 =\norm\cdot_{D_1}^2 +\norm\cdot_{D_2}^2$
(defined as in \eqref{eq:residue-norm} with $\sigma :=1$, $\vphi_L
:=\vphi_F$ and $\lcS[V,p] :=D_p$).
With such choice of representative, our goal is to prove that the
$L^{2}$-norm $\norm{s u_{1}}_{D_1, \vphi_M}^2 +\norm{s u_{2}}_{D_2,
  \vphi_M}^2$ is actually zero (where $\norm\cdot_{D_p, \vphi_M}$'s
are defined as in \eqref{eq:residue-norm} with $\vphi_L:=\vphi_F
+\vphi_M$).

\end{step}

\begin{step}[Obstruction for $\norm{s u_{1}}_{D_1, \vphi_M}^2 +\norm{s
      u_{2}}_{D_2, \vphi_M}^2$ from being zero]
  \label{item:expression-of-su-simple}

  We make use of the assumption $s \alpha=0$ and the \v
  Cech--Dolbeault isomorphism to re-express $\norm{s u_{1}}_{D_1,
    \vphi_M}^2 +\norm{s u_{2}}_{D_2, \vphi_M}^2$ as follows.

  \begin{itemize}
  \item Take $\alpha_{p;\:ij}   \in H^{0}(V_{ij} \cap D_p ,
    K_{D_{p}}\otimes F)$  
    for every open set $V_{ij} :=V_i \cap V_j$ with $i,j \in I$ and
    $V_i \cap V_j \cap D_p\neq \emptyset$ such that the family
    $\{\alpha_{p;\:ij}\}_{i,j \in I}$ is a \v Cech cocycle
    representing $u_{p}$ via the \v Cech--Dolbeault isomorphism on
    $D_p$.
    It follows that there exists an $L^2$ section $v_{p,(2)}$ of
    $K_{D_p} \otimes \res F_{D_p}$ on $D_p$ with respect to
    $\norm\cdot_{D_p}$ such that (under Einstein summation
    convention) 
    \begin{equation*}
      u_p
      \overset{\text{\eqref{eq:Cech-Dolbeault-isom}}}= \:
      \dbar v_{p,(2)} -\dbar\rho^j \cdot \rho^i \:\alpha_{p;\:ij} \; .
    \end{equation*}

  \item Take $f_{ij} \in H^{0}(V_{ij} , K_X \otimes D
    \otimes F \otimes \defidlof{D_1 \cap D_2})$ for $i,j \in I$
    satisfying 
    \begin{equation*}
      \Res^1\paren{f_{ij}}
      :=\paren{\PRes[D_1](\frac{f_{ij}}{\sect_D}) \:,\:
        \PRes[D_2](\frac{f_{ij}}{\sect_D})} 
      = \paren{\alpha_{1;\:ij} , \alpha_{2;\:ij}} 
    \end{equation*}
    whose existence are guaranteed by the surjectivity of the
    residue isomorphism $\Res^1$ on Stein open sets such that
    \begin{equation*}
      \renewcommand{\objectstyle}{\displaystyle}
      \xymatrix@C=1em@R=0.8em{
        {K_X \otimes D \otimes \frac{\defidlof{D_1 \cap
              D_2}}{\defidlof{D}}} 
        \ar@{}[r]|-*+{=}
        \ar@{}[d]|*[left]{\in}
        &{K_X \otimes D \otimes \faidlof|1|/|0|*}
        \ar[rr]^-{\Res^1}_-{\isom}
        &&{K_X \otimes D \otimes \residlof|1|*}
        \ar@{}[r]|-*+{=}
        &{K_{D_1} \oplus K_{D_2}}
        \ar@{}[d]|(.57)*[left]{\in}
        \\
        *+/r 3em/{f_{ij} \bmod \defidlof{D}}
        \ar@{|->}[rrr]
        &&&*+/l 6em/{}
        &*-{\paren{\alpha_{1;\:ij} , \alpha_{2;\:ij}} \; .}
      }
    \end{equation*}
    It is easy to see that $\set{f_{ij} \bmod \defidlof{D}}_{i,j \in
      I}$ is a \v Cech cocycle whose cohomology class in $\cohgp
    1[D]{\logKX \otimes \frac{\defidlof{D_1 \cap
          D_2}}{\defidlof{D}}}$ is mapped to $\alpha$ via $\tau$. 

  \item The assumption $s \alpha=0$ in $\cohgp 1[D]{K_D
      \otimes F \otimes M}$ guarantees the
    existence of $\lambda_{i} \in H^{0}(V_{i}, K_X \otimes D\otimes
    F \otimes M)$ for $i \in I$ such that
    \begin{equation*}
      s f_{ij} \equiv \lambda_j -\lambda_i \mod \defidlof{D}
      \quad\text{ on } V_{ij} \; .
    \end{equation*}
    Note that the coefficients of $\lambda_i$ need not lie in
    $\defidlof{D_1 \cap D_2}$ even though so do those of $f_{ij}$.
    By setting
    \begin{equation*}
      \rs*\lambda_{p;\:i} := \PRes[D_p](\frac{\lambda_i}{\sect_D})
      \cdot \sect_{(p)}
      \quad\text{ on $V_i \cap D_p$ for } i\in I
      \text{ and } p = 1,2 \; ,
    \end{equation*}
    it then follows that
    \begin{equation*}
      s\alpha_{p;\:ij} \sect_{(p)} =\rs*\lambda_{p;\:j} -\rs*\lambda_{p;\:i}
      \quad\text{ on } V_{ij} \cap D_p \; .
    \end{equation*}
    Note that $\rs*\lambda_{p;\:i}$ is holomorphic on $V_i \cap D_p$
    (while $\PRes[D_p](\frac{\lambda_i}{\sect_D})$ may not be).
  \end{itemize}

  Since $u_{p}$ is harmonic with respect to $\vphi_F$ on $D_p$ and
  we have $\ibddbar\vphi_F \geq 0$ and $
  \ibddbar\vphi_M \leq C\ibddbar\vphi_F$ on $D_p$ for some constant
  $C > 0$ by assumption,
  Proposition \ref{prop:consequence-of-positivity} guarantees that  
  $su_p$ is harmonic with respect to $\vphi_F+\vphi_M$ on $D_p$,
  which is a consequence of Nakano's identity and Enoki's argument.
  It follows that $\iinner{s \dbar v_{p;(2)}}{su_p}_{D_p, \vphi_M}
  =\iinner{\dbar\paren{s v_{p;(2)}}}{su_p}_{D_p, \vphi_M} = 0$.
  Summarizing the above discussion, it follows that
  \begin{align*}
    \norm{s u_{p}}_{D_p, \vphi_M}^2 
    &= -\sum_{i,j\in I}\iinner{\dbar\rho^{j} \cdot \rho^i \:s
      \alpha_{p;\:ij} \:}{\:s u_p}_{D_p, \vphi_M}\\
    &= -\sum_{i,j\in I}\iinner{\dbar\rho^{j} \cdot \rho^i \:s
      \alpha_{p;\:ij} \sect_{(p)} \:}{\:s u_p \sect_{(p)}}_{D_p, \vphi_M+\phi_{(p)}}\\
    &= -\sum_{i,j\in I}\iinner{\dbar\rho^{j} \cdot \rho^i
      \paren{\rs*\lambda_{p;\:j}- \rs*\lambda_{p;\:i}} \:}
      {\:s u_p \sect_{(p)}}_{D_p, \vphi_M+\phi_{(p)}}\\
    &= -\sum_{j\in I}\iinner{\dbar\paren{\rho^j
      \rs*\lambda_{p;\:j}} \:}{\: s u_p \sect_{(p)}}_{D_p,
      \vphi_M+\phi_{(p)}}
      =: -\iinner{\dbar v_{p;(\infty)} }{ s u_p \sect_{(p)}}_{D_p,
      \vphi_M+\phi_{(p)}}
      \; .
  \end{align*}
  The notation $v_{p;(\infty)} :=\sum_{j\in I} \rho^j
  \rs*\lambda_{p;\:j}$ is used for the consistency with the notation in
  Proposition \ref{prop:res-formula-dbar-exact-dot-harmonic}.

  The residue computation in Proposition
  \ref{prop:res-formula-dbar-exact-dot-harmonic} further brings the
  expression of $\norm{s u_{p}}_{D_p, \vphi_M}^2$ for each $p=1,2$ to an inner
  product on $D_1 \cap D_2$.
  As $\lcS|2|[b] :=D_1 \cap D_2$ has only $1$ component, the index set $\Iset|2|
  =\set{b}$ is a singleton.
  Moreover, the general different $\Diff_{D_1 \cap D_2}(D) =\Diff_b(D)$ is
  trivial, so we choose its canonical section and the corresponding
  potential such that $\sect_{(b)} \equiv 1$ and $\phi_{(b)} \equiv 0$
  (and $\psi_{(b)} \equiv -1$) on $D_1 \cap D_2$.
  Let $\PRes[\lcS|2|[b] | D_p]$ be the Poincar\'e residue map from
  $D_p$ to $D_1 \cap D_2$.
  We fix the sign convention such that
  \begin{equation*}
    \rs*\lambda_{b;\:i}
    =\frac{\rs*\lambda_{b;\:i}}{\sect_{(b)}}
    :=\PRes[\lcS|2|[b]](\frac{\lambda_i}{\sect_D})
    \begin{aligned}[t]
      &= \PRes[\lcS|2|[b] | D_1] \circ
      \PRes[D_1](\frac{\lambda_i}{\sect_D})
      \\
      &=\PRes[\lcS|2|[b] | D_1](\frac{\rs*\lambda_{1;\:i}}{\sect_{(1)}})
    \end{aligned}
    \begin{aligned}[t]
      &=-\PRes[\lcS|2|[b] | D_2] \circ
      \PRes[D_2](\frac{\lambda_i}{\sect_D})
      \\
      &=-\PRes[\lcS|2|[b] | D_2](\frac{\rs*\lambda_{2;\:i}}{\sect_{(2)}})
    \end{aligned} \; .
  \end{equation*}
  Following the computation in Proposition
  \ref{prop:res-formula-dbar-exact-dot-harmonic}, we obtain
  \begin{align*}
    &~\norm{s u_{1}}_{D_1, \vphi_M}^2 +\norm{s u_{2}}_{D_2,
      \vphi_M}^2
    \\
    =&~-\iinner{\dbar v_{1;(\infty)} }{ s u_1 \sect_{(1)}}_{D_1,
       \vphi_M+\phi_{(1)}}
       -\iinner{\dbar v_{2;(\infty)} }{ s u_2 \sect_{(2)}}_{D_2,
       \vphi_M+\phi_{(2)}}
    \\
    =&~
       \begin{multlined}[t]
         \sum_{i\in I}\iinner{\rho^i \PRes[\lcS|2|[b] |
           D_1](\frac{\rs*\lambda_{1;\:i}}{\sect_{(1)}}) }{
           \: s\:\PRes[\lcS|2|[b] | D_1](\idxup{\diff\psi_{(1)}}. u_1)
         }_{D_1 \cap D_2, \vphi_M} \\
         +\sum_{i\in I}\iinner{\rho^i
           \PRes[\lcS|2|[b] |
           D_2](\frac{\rs*\lambda_{2;\:i}}{\sect_{(2)}}) }{
           \: s\:\PRes[\lcS|2|[b] | D_2](\idxup{\diff\psi_{(2)}}. u_2)
         }_{D_1 \cap D_2, \vphi_M}
       \end{multlined}
    \\
    =&~\iinner{\sum_{i\in
       I}\rho^i\rs*\lambda_{b;\:i} \:}{\: s\:\paren{
       \PRes[\lcS|2|[b] | D_1](\idxup{\diff\psi_{(1)}}. u_1)
       -\PRes[\lcS|2|[b] | D_2](\idxup{\diff\psi_{(2)}}. u_2)
       }}_{D_1 \cap D_2, \vphi_M}
    \\
    =:&~\iinner{v_{b;(\infty)}}{s w_b}_{D_1 \cap D_2, \vphi_M} \; ,
  \end{align*}
  which is the desired expression.

  It is shown below that
  \begin{equation} \label{eq:w-prelim-formula}
    w_b :=\PRes[\lcS|2|[b] | D_1](\idxup{\diff\psi_{(1)}}. u_1)
    -\PRes[\lcS|2|[b] | D_2](\idxup{\diff\psi_{(2)}}. u_2)
  \end{equation}
  is actually $0$ on $D_1 \cap D_2$, which will then conclude the proof.

\end{step}

\begin{step}[$w_b$ being holomorphic and thus {$w_b \in \cohgp 0[D_1
    \cap D_2]{K_{D_1 \cap D_2} \otimes F}$}]
  
  We prove that $\dbar w_b = 0$ on $\lcS|2|[b] :=D_1 \cap D_2$ by a
  direct computation given in Section \ref{subsec:harmonic}.
  Indeed, it suffices to show that each summand $\PRes[\lcS|2|[b] |
  D_p](\idxup{\diff\psi_{(p)}}. u_p)$ for $p=1,2$ in $w_b$ is
  $\dbar$-closed.
  The computations are identical, so it suffices to consider $p=1$.

  On an admissible open set $V$ such that $D_p  \cap V =\set{z_p =
    0}$ for $p=1,2$ and $\lcS|2|[b] \cap V = \set{z_1 = z_2 = 0} =
  D_1 \cap \set{z_2 = 0}$,
  we have
  \begin{equation*}
    \diff\psi_{(1)} =\frac{dz_2}{z_2} -\diff\sm\vphi_{(1)}
    \quad\text{ on } V \; .
  \end{equation*}
  By writing
  \begin{equation*}
    \idxup{dz_2}. u_1 =: dz_2 \wedge \paren{\idxup{dz_2}. \rs*u_{1,2}}
    \quad\text{ on } D_1 \cap V \; ,
  \end{equation*}
  where $\rs*u_{1,2}$ is a $(n-2,1)$-form on $D_1 \cap V$, we see
  that
  \begin{equation*}
    \PRes[\lcS|2|[b] | D_1](\idxup{\diff\psi_{(1)}}. u_1)
    =\PRes[\set{z_2 = 0}](\frac{\idxup{dz_2} .u_1}{z_2})
    =\parres{\idxup{dz_2}. \rs*u_{1,2}}_{\lcS|2|[b]}
    \quad\text{ on } D_1 \cap D_2 \cap V \; .
  \end{equation*}
  Therefore, it suffices to check that $\idxup{dz_2}. u_1$ is
  $\dbar$-closed on $D_1 \cap V$.
  As $u_1$ is harmonic and $\ibddbar\vphi_F \geq 0$, we have
  $\nabla^{(0,1)} u_1 = 0$ by Proposition
  \ref{prop:consequence-of-positivity} and Lemma
  \ref{lem:commutator-dbar-ctrt} yields the desired result (with $z_2$
  in place of $\vphi$ in the lemma).

\end{step}

\begin{step}[$w_b = 0$ and conclusion of the proof]
  \label{step:pf:use_u-ortho-w-simple}
We prove that $w_b =0$ using the assumption $(u_{1},u_{2}) \in
\paren{\ker \tau}^\perp$.
Consider the connecting morphism $\delta$ the long exact sequence 
\begin{equation*}
  \xymatrix@R=0.3cm@C=1.5em{
    {\to \cohgp 0[D_{1}\cap D_{2}]{K_{D_1 \cap D_2} \otimes F}} \ar[r]^-{\delta}
    &
    {\bigoplus_{p=1}^{2} \cohgp 1[D_{p}]{K_{D_p}\otimes F}} \ar[r]^-{\tau}  
    &
    {\cohgp 1[D]{K_D \otimes F}  \to} \; . 
  } 
\end{equation*}
Note that $\delta w_b \in \ker\tau$.

We compute $\delta w_b$ via the \v Cech--Dolbeault isomorphism.
Regard $w_b$ as a $0$-cocycle $\set{\rs \gamma_{b;\:i}}_{i \in I}$
given by $\rs \gamma_{b;\:i} :=\res{w_b}_{V_i}$.
Lift $\rs \gamma_{b;\:i}$ on $D_1 \cap D_2 \cap V_i$ to a section
$\gamma_i$ on $V_{i}$ via the isomorphism
$\frac{\holo_X}{\defidlof{D_1 \cap D_2}} = \faidlof|2|/|1|*
\xrightarrow[\isom]{\Res^2} \residlof|2|*$ such that
\begin{equation*}
  \Res^2\paren{\gamma_i}
  =\PRes[\lcS|2|[b]](\frac{\gamma_i}{\sect_D})
  =\frac{\rs*\gamma_{b;\:i}}{\sect_{(b)}} =\rs*\gamma_{b;\:i} \; .
\end{equation*}
Then $\delta w_b$ is represented by the $1$-cocycle
\begin{equation*}
  \delta\set{\gamma_i \bmod \defidlof{D_1 \cap D_2}}_{i \in I}
  =\set{(\delta  \gamma)_{ij} \bmod\defidlof{D}}_{i,j \in I}
  =\set{\gamma_{j} -\gamma_i  \bmod\defidlof{D}}_{i,j \in I} \; .
\end{equation*}
Note that $ \gamma_{j} -\gamma_i$ belongs to $\defidlof{D_1 \cap
  D_2}$, so $ \gamma_{j} -\gamma_i  \bmod\defidlof{D}$ can be realized
via the isomorphism $\frac{\defidlof{D_1 \cap D_2}}{\defidlof{D}}
=\faidlof|1|/|0|* \xrightarrow[\isom]{\Res^1} \residlof|1|*$ as
\begin{align*}
  \Res^1\paren{\gamma_{j} -\gamma_i}
  &=\paren{\PRes[D_1](\frac{\gamma_{j} -\gamma_i}{\sect_D})
    \: ,\:
    \PRes[D_2](\frac{\gamma_{j} -\gamma_i}{\sect_D})
    } \\
  &=\paren{
    \frac{(\delta \rs\gamma_1)_{ij}}{\sect_{(1)}}
    \: , \:
    \frac{(\delta \rs\gamma_2)_{ij}}{\sect_{(2)}}
    }
    \in K_{D_1} \otimes \res F_{D_1} \oplus K_{D_2} \otimes \res F_{D_2} \; ,
\end{align*}
in which $\rs*\gamma_{p;\:i} := \PRes[D_p](\frac{\gamma_i}{\sect_D})
\cdot \sect_{(p)}$ for $p = 1,2$.
Therefore, via the \v Cech--Dolbeault isomorphism on each $D_p$, 
the component of $\delta w_b$ on $D_p$ can be represented by (under
Einstein summation convention) 
\begin{equation*}
  -\dbar\rho^j \cdot \rho^i
  \frac{\paren{\delta\rs*\gamma_p}_{ij}}{\sect_{(p)}}
  =-\frac{\dbar\rho^j \cdot\rs*\gamma_{p;\:j}}{\sect_{(p)}}
  =: -\frac{\dbar v'_{p;(\infty)}}{\sect_{(p)}}
\end{equation*}
(the notation $v'_{p;(\infty)} :=\sum_{i\in I}\rho^i
\rs*\gamma_{p;\:i}$ is set for the consistency with the notation in
Proposition \ref{prop:res-formula-dbar-exact-dot-harmonic}).
Recall the sign convention chosen in Step
\ref{item:expression-of-su-simple} such that
\begin{equation*}
  \rs*\gamma_{b;\:i}
  = \PRes[\lcS|2|[b] | D_1](\frac{\rs*\gamma_{1;\:i}}{\sect_{(1)}})
  =- \PRes[\lcS|2|[b] |
  D_2](\frac{\rs*\gamma_{2;\:i}}{\sect_{(2)}}) \; .
\end{equation*}
Then, from $(u_{1},u_{2}) \in \paren{\ker\tau}^\perp$ and $\delta w_b
\in \ker\tau$, we obtain
\begin{align*}
  0
  &=
  \iinner{
    -\frac{\dbar v'_{1;(\infty)}}{\sect_{(1)}}
  }{u_1}_{D_1}
  +\iinner{
    -\frac{\dbar v'_{2;(\infty)}}{\sect_{(2)}}
  }{u_2}_{D_2}
  \\
  &=
    \iinner{
    -\dbar v'_{1;(\infty)}
    }{u_1 \sect_{(1)}}_{D_1, \phi_{(1)}}
    +\iinner{
    -\dbar v'_{2;(\infty)}
    }{u_2 \sect_{(2)}}_{D_2, \phi_{(2)}}
  \\
  &\overset{\mathclap{\text{Prop.~\ref{prop:res-formula-dbar-exact-dot-harmonic}}}}=
    \quad\;\;
    \iinner{\rho^i \rs*\gamma_{b;\:i} \:}{\:
    \PRes[\lcS|2|[b] | D_1](\idxup{\diff\psi_{(1)}}. u_1)
    -\PRes[\lcS|2|[b] | D_2](\idxup{\diff\psi_{(2)}}. u_2)
    }_{D_1\cap D_2}
  \\
  &=\iinner{w_b}{w_b}_{D_1 \cap D_2}
    =\norm{w_b}_{D_1 \cap D_2}^2
    \; .
\end{align*}
This implies that $w_b=0$, finishing the proof for the case
$D=D_{1}+D_{2}$ and $q=1$. \qedhere
\end{step}
\end{proof}

\subsection{Remarks on the general case}
\label{subsec:n3}

There are two modifications to the proof in Section
\ref{sec:proof-of-simple-case} in order to handle the general case
worth mentioning here.
The first one is the replacement of the short exact sequence $0 \to
K_{D_1} \oplus K_{D_2} \to K_D \to K_{D_1 \cap D_2} \to 0$.
Take the case $D = D_1 + D_2 + D_3$, where $D_p = \set{z_p = 0}$ for
$p=1,2,3$ are the coordinate planes, for example.
Note that
\begin{equation*}
  \aidlof|3|* = \holo_X \;, \;\;
  \aidlof|2|* = \defidlof{D_1 \cap D_2 \cap D_3} \;, \;\;
  \aidlof|1|* = \smashoperator{\bigcap_{\substack{1 \leq p,q \leq 3 \\ p\neq q}}} \defidlof{D_p \cap D_q} \;
  \text{ and } \;
  \aidlof|0|* = \defidlof{D} 
\end{equation*}
in this case.
A natural choice of the short exact sequence to be considered is
\begin{equation*}
  \renewcommand{\objectstyle}{\displaystyle}
  \xymatrix@R=2.5em{
    0 \ar[r]
    &{K_X \otimes D \otimes \smash{\faidlof|1|/|0|*}} \ar[r]
    \ar[d]^(0.47){\Res^1}_(0.47){\isom}
    &{K_X \otimes D \otimes \smash{\faidlof|3|/|0|*}} \ar[r]
    \ar@{=}[d]
    &{K_X \otimes D \otimes \faidlof|3|/|1|*} \ar[r]
    &0 \; .
    \\
    &{\smash{\bigoplus_{p = 1}^3}\:K_{D_p}} \ar[r]
    &{K_D} 
    &
  }
\end{equation*}
In the previous case, we are taking advantage of the fact that the
$L^2$ Dolbeault isomorphism and the harmonic theory are valid on the
cohomology groups of the sheaves on both the left- and
right-hand-sides of the short exact sequence, so that the
corresponding injectivity statement can be proved on each side in
the spirit of Enoki, which in turn leads to the injectivity theorem
for the cohomology groups of the middle sheaf (twisted by $F$).
In the current case, they are valid only on the left-hand-side (on
each $D_p$).
We are thus led to determine whether the injectivity statement for
the sheaf on the right-hand-side holds true.
It is then apparent that we should consider
\begin{equation*}
  \renewcommand{\objectstyle}{\displaystyle}
  \xymatrix@R=2.5em{
    0 \ar[r]
    &{K_X \otimes D \otimes \smash{\faidlof|2|/|1|*}} \ar[r]
    \ar[d]_(0.47){\Res^2}^(0.47){\isom}
    &{K_X \otimes D \otimes \faidlof|3|/|1|*} \ar[r]
    &{K_X \otimes D \otimes \smash{\faidlof|3|/|2|*}} \ar[r]
    \ar[d]^-{\Res^3}_-{\isom}
    &0 \; ,
    \\
    &{\smash[t]{\bigoplus_{\substack{p,q = 1 \\ p\neq q}}^3} K_{D_p \cap D_q}} 
    &
    &{K_{D_1 \cap D_2 \cap D_3}}
  }
\end{equation*}
which, again, has the Dolbeault and harmonic theories valid on both
sides (on each lc center of $(X,D)$) of the short exact sequence.
The arguments in Section \ref{sec:proof-of-simple-case} can then be
employed to conclude the proof.
This illustrates the idea of the inductive arguments, which reduces
the question to the union of lower dimensional lc centers of $(X,D)$
in each step, to be employed in the general proof in Section
\ref{subsec:general}. 

Another modification to the proof in Section
\ref{sec:proof-of-simple-case} is that, when the claim in Theorem
\ref{thm:main} with $q > 1$ is considered, the section $w_b$
constructed as in \eqref{eq:w-prelim-formula} is then a $K_{\lcS+1[b]}
\otimes \res F_{\lcS+1[b]}$-valued $(0,q-1)$-form on some
$(\sigma+1)$-lc center $\lcS+1[b]$.
In order to prove that $w_b =0$ by following the arguments in the
previous case, we need not only to show that $w_b$ is
$\dbar$-closed, but also that it is harmonic.
This happens to be true and the computation for checking this claim
is given in Proposition \ref{prop:harmonic-residue} and Theorem
\ref{thm:residue-harmonic}.

\subsection{Proof of Theorem \ref{thm:main} in general}\label{subsec:general}



\renewcommand{\objectstyle}{\displaystyle}

Write
\begin{gather*}
  \aidlof* := \aidlof =\mtidlof{\vphi_F} \cdot \defidlof{\lcc+1'}
  =\defidlof{\lcc+1'} \; , \quad
  \residlof* := \residlof \isom \faidlof/-1*  \\
  \text{and } \quad \spH{\sheaf F}
  :=\cohgp q[X]{\logKX \otimes \sheaf F} 
\end{gather*}
for convenience.
Recall that
\begin{equation*}
  K_D = K_X \otimes D \otimes \faidlof|\sigma_{\mlc}|/|0|* \; ,
\end{equation*}
and the inclusions between adjoint ideal sheaves induce the short exact
sequences
\begin{equation*} 
  \xymatrix{
    0 \ar[r]
    & {\faidlof/|\rho|*} \ar[r]
    & {\faidlof|\tau|/|\rho|*} \ar[r]
    & {\faidlof|\tau|/*} \ar[r]
    & 0
  } \quad\text{ for } 0 \leq \rho \leq \sigma \leq \tau \; .
\end{equation*}
One is thus led to consider the commutative diagram

\begin{equation} \label{eq:commut-diagram_sing-Fujino-conj}
  \begin{gathered}
    \xymatrix@R=0.85cm@C+0.75cm{
      {\vdots} \ar[d]
      & {\vdots} \ar[d]
      & {\vdots} \ar[d]
      \\
      {\spH{\faidlof-1/|0|*}} \ar[d] \ar@{=}[r]
      & {\spH{\faidlof-1/|0|*}} \ar[r]^-{\otimes s}
      \ar[d]_-{\iota_{\sigma-1}} \ar[dr]|-*+{\mu_{\sigma-1}}
      & {\spH M{\faidlof-1/|0|*}} \ar[d]
      \\
      {\spH{\faidlof/|0|*}} \ar[d] \ar[r]^-{\iota_{\sigma}}
      \ar@/_1.9pc/[rr]|(.65)*+{\mu_{\sigma}}
      & {\spH{\faidlof|\sigma_{\mlc}|/|0|*}}
      \ar[d]|(.38)*+<3pt>{ }
      \ar[r]^-{\otimes s}
      & {\spH M{\faidlof|\sigma_{\mlc}|/|0|*}}
      \ar[d]
      \\
      {\spH{\residlof*}} \ar[d] \ar[r]^-{\tau_\sigma}
      \ar@/_1.65pc/[rr]+<-39pt,-15pt>|(.67)*+{\nu_\sigma}
      & {\spH{\faidlof|\sigma_{\mlc}|/-1*}} \ar[d]|(.52)*+<3pt>{}
      \ar[r]^-{\otimes s}
      & {\spH M{\faidlof|\sigma_{\mlc}|/-1*}} \ar[d] \\
      {\vdots} & {\vdots} & {\vdots} }
  \end{gathered}
\end{equation}
for $\sigma =2,\dots,\sigma_{\mlc}$, in which the columns are exact,
$\iota_\sigma$ and $\tau_\sigma$ are induced from the inclusions
between adjoint ideal sheaves, and $\mu_\sigma$ (resp.~$\nu_\sigma$)
is the composition of $\iota_\sigma$ (resp.~$\tau_\sigma$) with the
map induced from the multiplication map $\otimes s$.
The statement in Theorem \ref{thm:main} is proved if one shows that
$\ker\mu_{\sigma_{\mlc}} = \ker\iota_{\sigma_{\mlc}} = 0$
($\iota_{\sigma_{\mlc}}$ is the identity map).
Note that $\mu_1 =\nu_1$ and $\iota_1 =\tau_1$.
Following the argument in \cite{Chan&Choi_injectivity-I}*{Thm.~1.3.2}, since
$\ker\mu_{\sigma-1} =\ker\iota_{\sigma-1}$ and $\ker\nu_\sigma
=\ker\tau_\sigma$ together imply $\ker\mu_{\sigma}
=\ker\iota_{\sigma}$ via a diagram-chasing argument, to prove Theorem
\ref{thm:main}, it suffices to show the following theorem.

\begin{thm} \label{thm:ker-nu=ker-tau}
  $\ker\nu_\sigma =\ker\tau_\sigma$ for all $\sigma =1, \dots, \sigma_{\mlc}$.
\end{thm}

\begin{remark} \label{rem:general-commut-diagram}
  When $\aidlof|0|*$ in the commutative diagram
  \eqref{eq:commut-diagram_sing-Fujino-conj} is replaced by $0$ (which
  can be considered as $\aidlof|-1|*$), the setup is reduced to the one in
  \cite{Chan&Choi_injectivity-I}*{Thm.~1.3.2}, which states that
  Theorem \ref{thm:ker-nu=ker-tau} together with the result in
  \cite{Matsumura_injectivity-lc}*{Thm.~1.6} or
  \cite{Chan&Choi_injectivity-I}*{Thm.~1.2.1} implies that Fujino's
  conjecture is true.
  As a matter of fact, the proof of Theorem \ref{thm:ker-nu=ker-tau}
  can also be adapted to the case $\sigma = 0$ (with $\aidlof|-1|* =
  0$ and $\residlof|0|* =D^{-1} \isom \defidlof{D} =\aidlof|0|*$) which recovers the result in
  \cite{Matsumura_injectivity-lc}*{Thm.~1.6} as well as
  \cite{Chan&Choi_injectivity-I}*{Thm.~1.2.1}.
  Furthermore, by replacing $\aidlof|0|*$ by $\aidlof|\sigma_0-1|*$
  and $\aidlof|\sigma_{\mlc}|*$ by $\aidlof|\sigma'|*$ for any $0 <
  \sigma_0 \leq \sigma'$ and letting $\sigma$ vary within the range
  $\sigma_0 < \sigma \leq \sigma'$ in the diagram
  \eqref{eq:commut-diagram_sing-Fujino-conj}, one sees that the proof
  of Theorem \ref{thm:ker-nu=ker-tau} guarantees the statement of
  Theorem \ref{thm:main} but with $K_D$ replaced by $K_X \otimes D
  \otimes \faidlof|\sigma'|/|\sigma_0 -1|*$.
\end{remark}

\begin{proof}
  The proof consists of the following steps.
  \begin{enumerate}[label=\textbf{Step \Roman*:}, ref=\Roman*,
    leftmargin=0pt, labelsep=*, widest=VI, itemindent=*, align=left,
    itemsep=1.5ex]
  \item Make use of the $L^2$ Dolbeault and harmonic theory available on
    $\spH{\residlof*}$.

    Write $\lcc' =\bigcup_{p \in \Iset} \lcS$ as the union of
    $\sigma$-lc centers $\lcS$ of $(X,D)$.  Notice that $\residlof*$,
    hence $\spH{\residlof*}$, has a decomposition as a direct sum
    which yields
    \begin{equation*}
      \spH{\residlof*}
      =\bigoplus_{p \in \Iset} \cohgp q[\lcS]{K_{\lcS}
        \otimes F \otimes \mtidlof<\lcS>{\vphi_F}}
      =\bigoplus_{p \in \Iset} \cohgp q[\lcS]{K_{\lcS} \otimes F}
    \end{equation*}
    such that the $L^2$ Dolbeault isomorphism and harmonic theory are
    valid for the cohomology group in each summand.
    Take the (squared) residue norm
    $\norm\cdot_{\lcc'}^2 = \sum_{p \in \Iset} \norm\cdot_{\lcS}^2$ as
    the $L^2$ norm on $\spH{\residlof*}$.  Pick any element
    $u := (u_p)_{p \in \Iset} \in \spH{\residlof*}$ such that
    \begin{itemize}
    \item each $u_p$ is a harmonic form on $\lcS$ with respect to the
      given norm $\norm\cdot_{\lcS}$ and
    
    \item $u \in \ker\nu_\sigma \cap \paren{\ker\tau_\sigma}^\perp$,
      where the orthogonal complement $\paren{\ker\tau_\sigma}^\perp$
      of $\ker\tau_\sigma$ is taken with respect to the residue norm
      $\norm\cdot_{\lcc'}$.
    \end{itemize}
    The theorem is proved if it is shown that $u_p = 0$ for all
    $p \in \Iset$.

  \item \label{item:express-su-in-residue-norm}
    Obtain an expression of $\norm{su}_{\lcc'}^2$ using the
    assumption $u \in \ker\nu_\sigma$ and the \v Cech--Dolbeault
    isomorphism.

    Let $\cvr V :=\set{V_i}_{i \in I}$ be a locally finite cover of
    $X$ by admissible open sets with respect to
    $(\vphi_F,\vphi_M,\psi_D)$ and let $\set{\rho^i}_{i \in I}$ be a
    partition of unity subordinate to $\cvr V$.
    Their notations are abused to mean also their induced cover and
    partition of unity on $\lcc'$ for any $\sigma \geq 0$.
    For any choice of indices $\idx 0,q \in I$, write $V_{\idx 0.q}
    :=V_{i_0} \cap V_{i_1} \dotsm \cap V_{i_q}$ as usual.
  
    Through the \v Cech--Dolbeault isomorphism, every (cohomology
    class of) $u_p$ is represented by a \v Cech $q$-cocycle
    $\set{\alpha_{p; \:\idx 0.q}}_{\idx 0,q \in I}$ such that (under
    the Einstein summation convention on the indices $\idx 0,q$)
    \begin{equation*}
      u_p
      \overset{\text{\eqref{eq:Cech-Dolbeault-isom}}}=
      \:\dbar v_{p;(2)}
      +(-1)^q \:\underbrace{\dbar \rho^{i_{q}} \wedge \dotsm \wedge
        \dbar\rho^{i_1} \cdot \rho^{i_0} }_{=: \:
        \paren{\dbar\rho}^{\idx q.0}} \alpha_{p; \:\idx 0.q} \; ,
    \end{equation*}
    where $v_{p; (2)}$ is a $K_{\lcS} \otimes \res{F}_{\lcS}$-valued $(0,q-1)$-form
    on $\lcS$ with $L^2$ coefficients with respect to
    $\norm\cdot_{\lcS}$ and
    $\alpha_{p; \:\idx 0.q} \in K_{\lcS} \otimes \res F_{\lcS} \otimes
    \mtidlof<\lcS>{\vphi_F} =K_{\lcS} \otimes \res F_{\lcS}$ on
    $V_{\idx 0.q}$.
    In view of the
    residue short exact sequence, choose, for each choice of
    the multi-indices $(\idx 0,q)$, a section
    $f_{\idx 0,q} \in \logKX M \otimes \aidlof*$ on $V_{\idx 0.q}$
    such that
    \begin{equation*}
      \Res^\sigma(f_{\idx 0.q})
      =\paren{\alert{s} \alpha_{p; \:\idx 0.q}}_{p \in \Iset} 
    \end{equation*}
    (note that $V_{\idx 0.q}$ is Stein).  Considering the inclusion
    $\aidlof* \subset \aidlof|\sigma_{\mlc}|*$, write
    \begin{equation*}
      \eqcls{f_{\idx 0.q}} := \paren{f_{\idx 0.q} \bmod \aidlof-1*}
      \;\in \logKX M \otimes \faidlof|\sigma_{\mlc}|/-1*
      \quad\text{ on } V_{\idx 0.q} \; .
    \end{equation*}
    The collection $\set{\eqcls{f_{\idx 0.q}}}_{\idx 0,q \in I}$ is
    then a \v Cech $q$-cocycle representing $\nu_\sigma(u)$ in $\spH
    M{\faidlof|\sigma_{\mlc}|/-1*}$.
    The assumption $u \in \ker\nu_\sigma$ implies that this cocycle is
    a coboundary, that is,
    \begin{equation*}
      \set{\eqcls{f_{\idx 0.q}}}_{\idx 0,q \in I}
      =\delta\set{\eqcls{\lambda_{\idx 1.q}}}_{\idx 1,q \in I}
      =\set{\eqcls{\paren{\delta\lambda}_{\idx 0.q} }}_{\idx 0,q \in I}
    \end{equation*}
    for some $\lambda_{\idx 1.q} \in \logKX M \otimes
    \aidlof|\sigma_{\mlc}|*$ on $V_{\idx 1.q}$ (note that
    $\lambda_{\idx 1.q}$ need \emph{not} take values in $\aidlof*$
    even though the sections $f_{\idx 0.q}$ do), where
    $\paren{\delta\lambda}_{\idx 0.q}$ is given by the usual formula
    of \v Cech coboundary operator $\paren{\delta\lambda}_{\idx 0.q}
    :=\sum_{k =0}^q (-1)^k \lambda_{\idx 0[\dotsm \widehat{i_k}].q}$.
    Notice that $f_{\idx 0.q}$ and $\paren{\delta\lambda}_{\idx 0.q}$
    differ by an element in $\logKX M \otimes \aidlof-1*$ on $V_{\idx
      0.q}$.

    Thanks to the positivity $\ibddbar\vphi_F \geq 0$ and the bound
    $\ibddbar\vphi_M \leq C \ibddbar\vphi_F$ for some
    constant $C > 0$ on each $\lcS$, the product $s u_p$ is harmonic
    with respect to $\norm\cdot_{\lcS}$ (Proposition
    \ref{prop:consequence-of-positivity}), so $\iinner{s u_p}{s
      \:\dbar  v_{p;(2)}}_{\lcS} = \iinner{s u_p}{\dbar \paren{s
        v_{p;(2)}}}_{\lcS} = 0$ for every $p \in\Iset$.
    It follows that
    \begin{align*}
      \norm{su}_{\lcc'}^2
      =\sum_{p\in \Iset} \norm{su_p}_{\lcS}^2
      =&~(-1)^q \sum_{p\in \Iset} \sum_{\idx 0,q \in I} \iinner{\paren{\dbar\rho}^{\idx q.0}
        \:s \alpha_{p; \:\idx 0.q}}{\: s u_p}_{\lcS}
      \\
      =&~(-1)^q \sum_{p\in \Iset} \sum_{\idx 0,q \in I} \iinner{
         s \alpha_{p; \:\idx 0.q}
         }{\:\idxup{\diff\rho},[\idx 0.q].  s u_p}_{\lcS}
         \; ,
    \end{align*}
    where $\idxup{\diff\rho},[\idx 0.q].  \cdot $ is the adjoint
    of $\paren{\dbar\rho}^{\idx q.0} \cdot$.
    As in Step \ref{item:expression-of-su-simple} in Section
    \ref{sec:proof-of-simple-case}, the desired expression can be
    obtained by substituting
    \begin{equation*}
      s\alpha_{p; \:\idx 0.q} =\PRes[\lcS](\frac{f_{\idx
          0.q}}{\sect_D})
      =\PRes[\lcS](\frac{\paren{\delta\lambda}_{\idx
          0.q}}{\sect_D})
      =\frac{\paren{\delta \rs*\lambda_p}_{\idx 0.q}}{\sect_{(p)}}
      \; ,
    \end{equation*}
    where $\rs*\lambda_{p; \:\idx 1.q}
    :=\PRes[\lcS](\frac{\lambda_{\idx 1.q}}{\sect_D}) \cdot
    \sect_{(p)}$.
    For the sake of illustration, an alternative approach via a
    direct residue computation is presented here.
    Note that $\paren{\idxup{\diff\rho},[\idx 0.q].  s u_p}_{p \in
      \Iset} \in \logKX M \otimes \smooth_{X\:c\,*} \cdot\residlof*$ on $V_{\idx
      0.q}$, so it has a preimage $h^{\idx 0.q} \in \logKX M \otimes
    \smooth_{X\:c\,*} \cdot \aidlof*$ of $\Res^\sigma$ (considered as a
    $\smooth_{X\:c\,*}$-homomorphism).
    Fix such preimage on each open set $V_{\idx 0.q}$.
    From the direct computation of the residue function, it follows
    that
    \begin{align*}
      (-1)^q \:\norm{su}_{\lcc'}^2
      =&~\sum_{p\in \Iset} \sum_{\idx 0,q \in I} \iinner{
         s \alpha_{p; \:\idx 0.q}
         }{\:\idxup{\diff\rho},[\idx 0.q].  s u_p}_{\lcS}
      \\
      \xleftarrow{\eps \tendsto 0^+}
       &~\smashoperator[l]{\sum_{\idx 0,q \in I}} \eps
         \int_{\mathrlap{V_{\idx 0.q}}} \quad \frac{
         \inner{f_{\idx 0.q}}{h^{\idx 0.q}}
         \:e^{-\phi_D -\vphi_F -\vphi_M}
         }{\abs{\psi_D}^{\sigma +\eps}}
      \\
      \overset{\mathclap{\text{Prop.~\ref{prop:residue-product-X-to-lcS}}}}= \quad\;
       &~\smashoperator[l]{\sum_{\idx 0,q \in I}} \eps
         \int_{\mathrlap{V_{\idx 0.q}}} \quad \frac{
         \inner{\paren{\delta\lambda}_{\idx 0.q}}{h^{\idx 0.q}}
         \:e^{-\phi_D -\vphi_F -\vphi_M}
         }{\abs{\psi_D}^{\sigma +\eps}} +\BigO(\eps)
      \\
      \xrightarrow[\text{Prop.~\ref{prop:residue-product-X-to-lcS}}]{\eps \tendsto 0^+}
       &~\sum_{p\in \Iset} \sum_{\idx 0,q \in I}
         \iinner{\paren{\dbar\rho}^{\idx q.0}
         \paren{\delta\rs*\lambda_p}_{\idx 0.q}}{\: s u_p
         \sect_{(p)}}_{\lcS, \phi_{(p)}}
      \\
      =&~\sum_{p\in \Iset} \sum_{\idx 1,q \in I}
         \iinner{\dbar\rho^{i_q} \wedge \dotsm \wedge \dbar\rho^{i_1}
         \cdot\rs*\lambda_{p;\:\idx 1.q}}{\: s u_p \sect_{(p)}}_{\lcS,
         \phi_{(p)}}
      \\
      =&~(-1)^{q-1} \sum_{p\in \Iset} \iinner{\dbar v_{p;(\infty)}}{\: s u_p
         \sect_{(p)}}_{\lcS, \phi_{(p)}}
         \; ,\footnotemark
    \end{align*}%
    \footnotetext{
      If the $L^2$ Dolbeault isomorphism is valid for $\spH
      M{\faidlof|\sigma_{\mlc}|/-1*}$, such conclusion can be
      obtained simply from the fact that $\nu_\sigma(su)$ is
      represented by a smooth $\dbar$-exact form on $\lcc'$.
    }%
    where
    $v_{p; (\infty)} :=\sum_{\idx 1,q\in I} \dbar\rho^{i_q} \wedge \dotsm
    \wedge \dbar\rho^{i_2} \cdot \rho^{i_1} \rs*\lambda_{p;\:\idx 1.q}
    =\sum_{\idx 1,q\in I} \paren{\dbar\rho}^{\idx q.1}\rs*\lambda_{p; \:\idx 1.q}$. 

    The expression of $\norm{su}_{\lcc'}^2$ can be further transformed
    by an integration by parts using Proposition
    \ref{prop:res-formula-dbar-exact-dot-harmonic}, which becomes
    \begin{equation*}
      \norm{su}_{\lcc'}^2
      = \sigma\sum_{b \in \Iset+1}
      \iinner{v_{b;(\infty)} \:
      }{ \quad s \:
        \smashoperator{\sum_{p \in \Iset \colon \lcS+1[b] \subset
            \lcS}} \;\; \sgn{b:p}\:
        \PRes[\lcS+1[b] | \lcS](\idxup{\diff\psi_{(p)}}.  u_{p})
        \cdot \sect_{(b)}
      }_{\lcS+1[b], \phi_{(b)}} \; ,
    \end{equation*}
    where $v_{b;(\infty)} := \sum_{\idx 1,q \in I}
    \paren{\dbar\rho}^{\idx q.1} \rs*\lambda_{b; \:\idx 1.q}$ and
    $\rs*\lambda_{b; \:\idx 1.q}
    :=\PRes[\lcS+1[b]](\frac{\lambda_{\idx 1.q}}{\sect_D}) \cdot \sect_{(b)}$.

    Set
    \begin{equation}\label{eq-def-w}
      w_b := \smashoperator[r]{\sum_{p \in \Iset \colon \lcS+1[b] \subset
          \lcS}} \;\; \sgn{b:p}\:
      \PRes[\lcS+1[b] | \lcS](\idxup{\diff\psi_{(p)}} . u_{p})
      \; .
    \end{equation}
    It suffices to show that $w_b = 0$ on $\lcS+1[b]$ for each $b
    \in\Iset+1$ to conclude the proof.

  \item Show that $w_b$ is harmonic with respect to
    $\res{\vphi_F}_{\lcS+1[b]}$ (and $\res{\omega}_{\lcS+1[b]}$) on
    $\lcS+1[b]$ for all $b \in \Iset+1$ and thus $\paren{w_b}_{b
      \in\Iset+1}$ represents a class in $\spH/q-1/{\residlof+1*}$.

    {
      \newcommand{\lcSb}{\lcS+1[b]}

      To see that $w_b$ is $\dbar$-closed on $\lcSb$, it suffices to
      show that $\PRes[\lcS+1[b] | \lcS](\idxup{\diff\psi_{(p)}}. 
      u_{p})$ is $\dbar$-closed for all $p\in\Iset$ such that $\lcSb
      \subset \lcS$.
      Take any admissible open set $V$ such that $V \cap \lcSb
      \neq\emptyset$ and a holomorphic coordinate system such that
      $\sect_{(p)} = z_{p(\sigma+1)} \dotsm z_{p(\sigma_V)}$ on $V$.
      Suppose $\lcSb \cap V = \lcS \cap \set{z_{p(k)} = 0}$ for some $k
      =\sigma +1, \dots, \sigma_V$.
      Recall that
      \begin{equation*}
        \diff\psi_{(p)} = \sum_{k'=\sigma+1}^{\sigma_V}
        \frac{dz_{p(k')}}{z_{p(k')}} - \diff\sm\vphi_{(p)} \quad\text{
          on } V \; .
      \end{equation*}
      By writing
      \begin{equation*}
        \idxup{dz_{p(k)}}.  u_p =: dz_{p(k)} \wedge
        \paren{\idxup{dz_{p(k)}}.  \rs u_{p,k}} \quad\text{ on }
        \lcS \cap V \; ,
      \end{equation*}
      in which $\rs u_{p,k}$ is a $(n-\sigma-1,q)$-form on $\lcS \cap
      V$, it follows that
      \begin{equation*}
        \PRes[\lcS+1[b] | \lcS](\idxup{\diff\psi_{(p)}}.  u_{p})
        =\PRes[\set{z_{p(k)}=0}](\frac{\idxup{dz_{p(k)}}. 
          u_p}{z_{p(k)}})
        =\parres{\idxup{dz_{p(k)}} . \rs u_{p,k}}_{\lcSb}
        \quad\text{ on } \lcSb \cap V \; .
      \end{equation*}
      It thus suffices to show that $\idxup{dz_{p(k)}}.  u_p$ is
      $\dbar$-closed on $\lcS \cap V$.
      Since $u_p$ is harmonic and $\ibddbar\vphi_F \geq 0$, it follows
      that $\dbar u_p = 0$ and $\nabla^{(0,1)}u_p = 0$ (Proposition
      \ref{prop:consequence-of-positivity}).
      Putting $u_p$ into $u$ and $z_{p(k)}$ into $\vphi$ in Lemma
      \ref{lem:commutator-dbar-ctrt}, one has
      $\dbar\paren{\idxup{dz_{p(k)}}.  u_p} = 0$ on $\lcS \cap V$.
      As a result, $w_b$ is $\dbar$-closed on $\lcSb$ and
      $\paren{w_b}_{b\in\Iset+1}$ therefore represents a class in
      $\spH/q-1/{\residlof+1*}$.

    }

    Furthermore, by Proposition \ref{prop:harmonic-residue} and
    Theorem \ref{thm:residue-harmonic} (with
    $\lcS$ in place of $X$, $\lcS+1[b]$ in place of $D_p$ and
    $\psi_{(p)}$ in place of $\psi_{D_p}$), $w_b$ is a
    $K_{\lcS+1[b]} \otimes \res{F}_{\lcS+1[b]}$-valued $(0,q-1)$-form on
    $\lcS+1[b]$ (not only a $\res{\conj\holoform_X^{q-1}}_{\lcS+1[b]}$-valued
    section) which is harmonic with respect to
    $\res{\vphi_F}_{\lcS+1[b]}$.

  \item \label{item:pf:use_u-ortho-w}
    Apply the assumption $u =(u_p)_{p\in\Iset} \in
    \paren{\ker\tau_\sigma}^{\perp}$ via the use of $w
    :=\paren{w_b}_{b \in \Iset+1} \in \spH/q-1/{\residlof+1*}
    =\bigoplus_{b \in\Iset+1} \cohgp{q-1}[\lcS+1[b]]{K_{\lcS+1[b]}
      \otimes F}$ in view of the commutative diagram
    \begin{equation*}
      \xymatrix@R-0.3cm{
        {\dotsm} \ar[r]
        & {\spH/q-1/{\residlof+1*}} \ar[r]^-{\delta}
        \ar[d]^-{\tau_{\sigma+1}}
        & {\spH{\residlof*}} \ar[r]
        \ar@{=}[d]
        & {\spH{\faidlof+1/-1*}} \ar[r]
        \ar[d]
        & {\dotsm}
        \\
        {\dotsm} \ar[r]
        & {\spH/q-1/{\faidlof|\sigma_{\mlc}|*}} \ar[r]
        & {\spH{\residlof*}} \ar[r]^-{\tau_\sigma}
        & {\spH{\faidlof|\sigma_{\mlc}|/-1*}} \ar[r]
        & {\dotsm}
      }
    \end{equation*}
    and conclude that $u_p = 0$ on $\lcS$ for each $p\in\Iset$.

    From the commutative diagram, one sees that $\delta w \in
    \ker\tau_\sigma$.
    In view of the isomorphism $\residlof+1* \isom \faidlof+1*$ and by
    following the procedures in Step
    \ref{item:express-su-in-residue-norm}, one obtains
    a $\logKX \otimes \aidlof+1*$-valued \v Cech $(q-1)$-cochain
    $\set{\gamma_{\idx 1.q}}_{\idx 1,q \in I}$ with respect to $\cvr
    V$ such that, when
    \begin{equation*}
      \paren{\alpha'_{b; \:\idx 1.q}}_{b\in\Iset+1}
      :=\paren{\frac{\rs*\gamma_{b; \:\idx
            1.q}}{\sect_{(b)}}}_{\mathrlap{b\in\Iset+1}} \quad\;
      :=\Res^{\sigma+1}(\gamma_{\idx 1.q})
      \in \smashoperator[r]{\prod_{b\in\Iset+1}} K_{\lcS+1[b]} \otimes
      \res{F}_{\lcS+1[b]} \paren{\lcS+1[b] \cap V_{\idx 1.q}}
    \end{equation*}
    (notation chosen for the consistency with those in Proposition
    \ref{prop:res-formula-dbar-exact-dot-harmonic}) and
    \begin{equation*}
      \eqcls{\gamma_{\idx 1.q}}
      := \paren{\gamma_{\idx 1.q} \bmod \aidlof*} \in \logKX \otimes
      \faidlof+1* \quad\text{ on } V_{\idx 1.q} \; ,
    \end{equation*}
    the collection $\set{\alpha'_{b;\:\idx 1.q}}_{\idx 1,q\in I}$
    is a \v Cech \emph{$(q-1)$-cocycle} representing (the class of)
    $w_b$ in $\cohgp{q-1}[\lcS+1[b]]{K_{\lcS+1[b]} \otimes F}$ for
    each $b \in\Iset+1$ such that
    \begin{equation*}
      w_b = \dbar v'_{b;(2)} +(-1)^{q-1} \:\underbrace{
        \dbar \rho^{i_{q}} \wedge \dotsm \wedge
        \dbar\rho^{i_2} \cdot \rho^{i_1} }_{=: \:
        \paren{\dbar\rho}^{\idx q.1}} \alpha'_{b;\:\idx 1.q}
      =: \dbar v'_{b;(2)} +(-1)^{q-1} \frac{v_{b;(\infty)}'}{\sect_{(b)}}
    \end{equation*}
    (again, notation chosen for the consistency with those in Proposition
    \ref{prop:res-formula-dbar-exact-dot-harmonic})
    for some $K_{\lcS+1[b]} \otimes \res{F}_{\lcS+1[b]}$-valued
    $(0,q-2)$-form $v'_{b;(2)}$ on $\lcS+1[b]$ with $L^2$ coefficients
    with respect to $\norm\cdot_{\lcS+1[b]}$, and the collection
    $\set{\eqcls{\gamma_{\idx 1.q}}}_{\idx 1,q \in I}$ is a \v Cech
    \emph{$(q-1)$-cocycle} representing (the class of) $w$ in
    $\spH/q-1/{\faidlof+1*} \xrightarrow[\isom]{\Res^{\sigma+1}}
    \spH/q-1/{\residlof+1*}$.
    The image $\delta w$ in $\spH{\residlof*}$ is then represented by
    \begin{equation*}
      \set{\Res^\sigma\paren{\paren{\delta\gamma}_{\idx 0.q}}}_{\idx 0,q \in
      I} \; ,
    \end{equation*}
    in which $\delta$ is the \v Cech boundary operator.
    Note that applying $\Res^\sigma$ to $\paren{\delta\gamma}_{\idx
      0.q}$ is valid as $\set{\eqcls{\gamma_{\idx 1.q}}}_{\idx 1,q \in
      I}$ is a cocycle and thus coefficients of
    $\paren{\delta\gamma}_{\idx 0.q}$ lie in $\aidlof*$.
    Set
    \begin{equation*}
      \rs*\gamma_{p;\:\idx 1.q} := \PRes[\lcS](\frac{\gamma_{\idx
            1.q}}{\sect_D}) \cdot \sect_{(p)} 
    \end{equation*}
    such that
    \begin{equation*}
      \Res^\sigma\paren{\paren{\delta\gamma}_{\idx 0.q}}
      =\paren{\frac{\paren{\delta\rs*\gamma_p}_{\idx
            0.q}}{\sect_{(p)}}}_{p \in \Iset}
      \in \prod_{p \in \Iset} K_{\lcS} \otimes \res F_{\lcS}
      \paren{\lcS \cap V_{\idx 0.q}} \; .
    \end{equation*}
    Note that 
    \begin{equation*}
      (-1)^q \paren{\dbar\rho}^{\idx q.0}
      \frac{\paren{\delta \rs*\gamma_p}_{\idx 0.q}}{\sect_{(p)}}
      =-
      \frac{\dbar\paren{\paren{\dbar\rho}^{\idx q.1} \rs*\gamma_{p;\:\idx 1.q}}}{\sect_{(p)}}
      =: - \:\frac{\dbar v_{p; (\infty)}'}{\sect_{(p)}}
      \quad\text{ on } \lcS
    \end{equation*}
    is a $\dbar$-closed form representing the class of $\res{\delta
      w}_{\lcS}$ (the component of $\delta w$ on $\lcS$) in $\cohgp
    q[\lcS]{K_{\lcS} \otimes F}$ via Dolbeault isomorphism.
    
    Therefore, from the assumption $u \in
    \paren{\ker\tau_\sigma}^\perp$ and taking into account the \v
    Cech--Dolbeault isomorphism \eqref{eq:Cech-Dolbeault-isom} and the
    fact that each $u_p$ is harmonic, one has
    \begin{align*}
      0 =\iinner{(-1)^{q-1} \delta w}{u}_{\lcc'}
      &=(-1)^q \sum_{p\in \Iset} \iinner{\frac{\dbar
          v_{p;(\infty)}'}{\sect_{(p)}}}{u_p}_{\lcS}
      =(-1)^q \sum_{p\in \Iset} \iinner{\dbar v_{p;(\infty)}'}{u_p
        \sect_{(p)}}_{\lcS, \phi_{(p)}}
      \\
      &\overset{\mathclap{\text{Prop.~\ref{prop:res-formula-dbar-exact-dot-harmonic}}}}=
        \quad \;\; (-1)^{q-1} \sigma
        \smashoperator{\sum_{b \in \Iset+1}} \iinner{v_{b;(\infty)}'}{w_b
        \sect_{(b)}}_{\lcS+1[b], \phi_{(b)}}
      \\
      &=\sigma \smashoperator{\sum_{b \in \Iset+1}}
        \iinner{
          \paren{w_b -\dbar v_{b;(2)}'} \sect_{(b)}
        }{
          w_b \sect_{(b)}
        }_{\lcS+1[b], \phi_{(b)}}
      \\
      &\overset{\mathclap{w_b \text{ harmonic}}}= \quad\;\;\;
        \sigma
        \smashoperator{\sum_{b \in \Iset+1}}
        \iinner{w_b}{w_b}_{\lcS+1[b]}
        =\sigma \norm{w}_{\lcc+1'}^2 \; .
    \end{align*}
    As a result, $w_b = 0$ for each $b\in\Iset+1$, thus $su_p = 0$
    (hence $u_p = 0$) for each $p\in\Iset$ by Step
    \ref{item:express-su-in-residue-norm}.
    This completes the proof. \qedhere
  \end{enumerate}
\end{proof}

\begin{remark} \label{rem:singular-vphi_F}
  When $\vphi_F$ and $\vphi_M$ have only neat analytic singularities
  such that $\vphi_F^{-1}(-\infty) \cup \vphi_M^{-1}(-\infty) \cup D$
  has only snc and $\vphi_F^{-1}(-\infty) \cup \vphi_M^{-1}(-\infty)$
  contains no irreducible components of $D$ (hence no lc centers of
  $(X,D)$), the proof is still valid when the K\"ahler metric $\omega$
  on $X$ is replaced by a complete metric on $X \setminus
  \paren{\vphi_F^{-1}(-\infty) \cup \vphi_M^{-1}(-\infty)}$ as
  described in \cite{Chan&Choi_injectivity-I}*{\S 2.2 item (4)}.
  See \cite{Chan&Choi_injectivity-I}*{\S 3.3} for the technical
  modifications required.
\end{remark}

\begin{remark} \label{rem:no-hard-Lefschetz}
  Notice that the refinement of hard Lefschetz theorem (see
  \cite{Matsumura_injectivity-lc}*{Thm.~1.7} or
  \cite{Chan&Choi_injectivity-I}*{Thm.~2.5.1}) is not used in this
  proof.
  It is used in previous works to show that $\frac{u}{\sect_D}$ is
  smooth for every harmonic $u$ representing a class in $\cohgp
  q[X]{\logKX\otimes \mtidlof<X>{\phi_D}}$.
  This argument can be replaced by using directly the isomorphism
  induced by $\holo_X \xrightarrow[\isom]{\otimes \sect_D} D \otimes
  \mtidlof<X>{\phi_D}$, or $\holo_{\lcS} \xrightarrow[\isom]{\otimes
    \sect_{(p)}} \Diff_p(D) \otimes \mtidlof<\lcS>{\phi_{(p)}}$, which
  is more relevant to this article (see also Lemma
  \ref{lem:su-harmonicity}).
  However, when $\vphi_F$ and $\vphi_M$ have neat analytic singularities
  as described in \cite{Chan&Choi_injectivity-I}*{\S 2.2}, the theorem
  is still needed to get certain control of the regularity of $u$ on the
  polar sets of $\vphi_F$ and $\vphi_M$ (see
  \cite{Chan&Choi_injectivity-I}*{Prop.~3.3.1}).
\end{remark}


\begin{bibdiv}
  \begin{biblist}
    \IfFileExists{references.ltb}{
      \bibselect{references}
    }{
      \bib{Ambro_quasi-log-var}{article}{
  author={Ambro, F.},
  title={Quasi-log varieties},
  journal={Tr. Mat. Inst. Steklova},
  volume={240},
  date={2003},
  number={Biratsion. Geom. Line\u {\i }n. Sist. Konechno Porozhdennye Algebry},
  pages={220--239},
  issn={0371-9685},
  translation={ journal={Proc. Steklov Inst. Math.}, date={2003}, number={1(240)}, pages={214--233}, issn={0081-5438}, },
  review={\MR {1993751}},
}

\bib{Ambro_injectivity}{article}{
  author={Ambro, Florin},
  title={An injectivity theorem},
  journal={Compos. Math.},
  volume={150},
  date={2014},
  number={6},
  pages={999--1023},
  issn={0010-437X},
  review={\MR {3223880}},
  doi={10.1112/S0010437X13007768},
}

\bib{Cao&Demailly&Matsumura}{article}{
  author={Cao, JunYan},
  author={Demailly, Jean-Pierre},
  author={Matsumura, Shinichi},
  title={A general extension theorem for cohomology classes on non reduced analytic subspaces},
  journal={Sci. China Math.},
  volume={60},
  date={2017},
  number={6},
  pages={949--962},
  issn={1674-7283},
  review={\MR {3647124}},
  doi={10.1007/s11425-017-9066-0},
}

\bib{Cao&Paun_LC-inj}{article}{
  author={Cao, Junyan},
  author={P\u aun, Mihai},
  title={$\partial \bar \partial $-lemmas and a conjecture of O. Fujino},
  eprint={arXiv:2303.16239 [math.AG]},
  date={2023},
}

\bib{Chan_on-L2-ext-with-lc-measures}{article}{
  author={Chan, Tsz On Mario},
  title={On an $L^2$ extension theorem from log-canonical centres with log-canonical measures},
  journal={Math. Z.},
  volume={301},
  date={2022},
  number={2},
  pages={1695--1717},
  issn={0025-5874},
  review={\MR {4418335}},
  doi={10.1007/s00209-021-02890-9},
  eprint={https://rdcu.be/cFDPA},
  arxiv={2008.03019 [math.CV]},
  note={Numbering of cited sections and theorems follows the arXiv version},
}

\bib{Chan_adjoint-ideal-nas}{article}{
  author={Chan, Tsz On Mario},
  title={A new definition of analytic adjoint ideal sheaves via the residue functions of log-canonical measures I},
  journal={J. Geom. Anal.},
  volume={33},
  date={2023},
  pages={Paper No. 279, 68 pp.},
  doi={10.1007/s12220-023-01314-w},
  eprint={https://rdcu.be/deUDt},
  arxiv={2111.05006 [math.CV]},
}

\bib{Chan&Choi_ext-with-lcv-codim-1}{article}{
  author={Chan, Tsz On Mario},
  author={Choi, Young-Jun},
  title={Extension with log-canonical measures and an improvement to the plt extension of Demailly-Hacon-P\u {a}un},
  journal={Math. Ann.},
  volume={383},
  date={2022},
  number={3-4},
  pages={943--997},
  issn={0025-5831},
  review={\MR {4458394}},
  doi={10.1007/s00208-021-02152-3},
  eprint={https://rdcu.be/cn5N6},
  arxiv={1912.08076 [math.CV]},
}

\bib{Chan&Choi_injectivity-I}{article}{
  author={Chan, Tsz On Mario},
  author={Choi, Young-Jun},
  title={On an injectivity theorem for log-canonical pairs with analytic adjoint ideal sheaves},
  arxiv={2205.06954 [math.CV]},
  date={2023},
  note={to appear in Trans. Amer. Math. Soc.},
}

\bib{Demailly}{webpage}{
  author={Demailly, Jean-Pierre},
  title={Complex analytic and differential geometry},
  note={OpenContent Book},
  url={https://www-fourier.ujf-grenoble.fr/~demailly/manuscripts/agbook.pdf},
  date={2012},
}

\bib{Donnelly&Xavier}{article}{
  author={Donnelly, Harold},
  author={Xavier, Frederico},
  title={On the differential form spectrum of negatively curved Riemannian manifolds},
  journal={Amer. J. Math.},
  volume={106},
  date={1984},
  number={1},
  pages={169--185},
  issn={0002-9327},
  review={\MR {729759}},
  doi={10.2307/2374434},
}

\bib{Enoki}{article}{
  author={Enoki, Ichiro},
  title={Kawamata-Viehweg vanishing theorem for compact K\"{a}hler manifolds},
  conference={ title={Einstein metrics and Yang-Mills connections}, address={Sanda}, date={1990}, },
  book={ series={Lecture Notes in Pure and Appl. Math.}, volume={145}, publisher={Dekker, New York}, },
  date={1993},
  pages={59--68},
  review={\MR {1215279}},
}

\bib{Esnault&Viehweg_book}{book}{
  author={Esnault, H\'{e}l\`ene},
  author={Viehweg, Eckart},
  title={Lectures on vanishing theorems},
  series={DMV Seminar},
  volume={20},
  publisher={Birkh\"{a}user Verlag, Basel},
  date={1992},
  pages={vi+164},
  isbn={3-7643-2822-3},
  review={\MR {1193913}},
  doi={10.1007/978-3-0348-8600-0},
}

\bib{Fujino_log-MMP}{article}{
  author={Fujino, Osamu},
  title={Fundamental theorems for the log minimal model program},
  journal={Publ. Res. Inst. Math. Sci.},
  volume={47},
  date={2011},
  number={3},
  pages={727--789},
  issn={0034-5318},
  review={\MR {2832805}},
  doi={10.2977/PRIMS/50},
}

\bib{Fujino_injectivity-II}{article}{
  author={Fujino, Osamu},
  title={A transcendental approach to Koll\'{a}r's injectivity theorem II},
  journal={J. Reine Angew. Math.},
  volume={681},
  date={2013},
  pages={149--174},
  issn={0075-4102},
  review={\MR {3181493}},
  doi={10.1515/crelle-2012-0036},
}

\bib{Fujino_vanishing-thms}{article}{
  author={Fujino, Osamu},
  title={Vanishing theorems},
  conference={ title={Minimal models and extremal rays}, address={Kyoto}, date={2011}, },
  book={ series={Adv. Stud. Pure Math.}, volume={70}, publisher={Math. Soc. Japan, [Tokyo]}, },
  date={2016},
  pages={299--321},
  review={\MR {3618264}},
  doi={10.2969/aspm/07010299},
  arxiv={1202.4200v2 [math.AG]},
}

\bib{Fujino_injectivity-hodge-theoretic}{article}{
  author={Fujino, Osamu},
  title={Injectivity theorems},
  conference={ title={Higher dimensional algebraic geometry---in honour of Professor Yujiro Kawamata's sixtieth birthday}, },
  book={ series={Adv. Stud. Pure Math.}, volume={74}, publisher={Math. Soc. Japan, Tokyo}, },
  date={2017},
  pages={131--157},
  review={\MR {3791211}},
  doi={10.2969/aspm/07410131},
  arxiv={1303.2404v3 [math.AG]},
}

\bib{Fujino_survey}{article}{
  author={Fujino, Osamu},
  title={On semipositivity, injectivity and vanishing theorems},
  conference={ title={Hodge theory and $L^2$-analysis}, },
  book={ series={Adv. Lect. Math. (ALM)}, volume={39}, publisher={Int. Press, Somerville, MA}, },
  date={2017},
  pages={245--282},
  review={\MR {3751293}},
}

\bib{Fujino&Matsumura}{article}{
  author={Fujino, Osamu},
  author={Matsumura, Shin-ichi},
  title={Injectivity theorem for pseudo-effective line bundles and its applications},
  journal={Trans. Amer. Math. Soc. Ser. B},
  volume={8},
  date={2021},
  pages={849--884},
  review={\MR {4324359}},
  doi={10.1090/btran/86},
  arxiv={1605.02284 [math.CV]},
}

\bib{Gongyo&Matsumura}{article}{
  author={Gongyo, Yoshinori},
  author={Matsumura, Shinichi},
  title={Versions of injectivity and extension theorems},
  language={English, with English and French summaries},
  journal={Ann. Sci. \'{E}c. Norm. Sup\'{e}r. (4)},
  volume={50},
  date={2017},
  number={2},
  pages={479--502},
  issn={0012-9593},
  review={\MR {3621435}},
  doi={10.24033/asens.2325},
  arxiv={1406.6132 [math.AG]},
}

\bib{Griffiths&Harris}{book}{
  author={Griffiths, Phillip},
  author={Harris, Joseph},
  title={Principles of algebraic geometry},
  series={Wiley Classics Library},
  note={Reprint of the 1978 original},
  publisher={John Wiley \& Sons, Inc., New York},
  date={1994},
  pages={xiv+813},
  isbn={0-471-05059-8},
  review={\MR {1288523}},
  doi={10.1002/9781118032527},
}

\bib{Kollar_injectivity}{article}{
  author={Koll\'{a}r, J\'{a}nos},
  title={Higher direct images of dualizing sheaves. I},
  journal={Ann. of Math. (2)},
  volume={123},
  date={1986},
  number={1},
  pages={11--42},
  issn={0003-486X},
  review={\MR {825838}},
  doi={10.2307/1971351},
}

\bib{Kollar_Sing-of-MMP}{book}{
  author={Koll\'{a}r, J\'{a}nos},
  title={Singularities of the minimal model program},
  series={Cambridge Tracts in Mathematics},
  volume={200},
  note={With a collaboration of S\'{a}ndor Kov\'{a}cs},
  publisher={Cambridge University Press, Cambridge},
  date={2013},
  pages={x+370},
  isbn={978-1-107-03534-8},
  review={\MR {3057950}},
  doi={10.1017/CBO9781139547895},
}

\bib{Matsumura_injectivity-survey}{article}{
  author={Matsumura, Shinichi},
  title={Injectivity theorems with multiplier ideal sheaves and their applications},
  conference={ title={Complex analysis and geometry}, },
  book={ series={Springer Proc. Math. Stat.}, volume={144}, publisher={Springer, Tokyo}, },
  date={2015},
  pages={241--255},
  review={\MR {3446761}},
  doi={10.1007/978-4-431-55744-9\_18},
}

\bib{Matsumura_injectivity}{article}{
  author={Matsumura, Shinichi},
  title={An injectivity theorem with multiplier ideal sheaves of singular metrics with transcendental singularities},
  journal={J. Algebraic Geom.},
  volume={27},
  date={2018},
  number={2},
  pages={305--337},
  issn={1056-3911},
  review={\MR {3764278}},
  doi={10.1090/jag/687},
  arxiv={1308.2033 [math.CV]},
}

\bib{Matsumura_rel-vanishing-w-nd}{article}{
  author={Matsumura, Shinichi},
  title={Variation of numerical dimension of singular hermitian line bundles},
  conference={ title={Geometric complex analysis}, },
  book={ series={Springer Proc. Math. Stat.}, volume={246}, publisher={Springer, Singapore}, },
  date={2018},
  pages={247--255},
  review={\MR {3923231}},
  doi={10.1007/978-981-13-1672-2\_19},
}

\bib{Matsumura_injectivity-lc}{article}{
  author={Matsumura, Shinichi},
  title={A transcendental approach to injectivity theorem for log canonical pairs},
  journal={Ann. Sc. Norm. Super. Pisa Cl. Sci. (5)},
  volume={19},
  date={2019},
  number={1},
  pages={311--334},
  issn={0391-173X},
  review={\MR {3923849}},
}

\bib{Matsumura_injectivity-Kaehler}{article}{
  author={Matsumura, Shinichi},
  title={Injectivity theorems with multiplier ideal sheaves for higher direct images under K\"{a}hler morphisms},
  journal={Algebr. Geom.},
  volume={9},
  date={2022},
  number={2},
  pages={122--158},
  issn={2313-1691},
  review={\MR {4429015}},
  doi={10.14231/ag-2022-005},
  arxiv={1607.05554v2 [math.CV]},
}

\bib{Ohsawa&Takegoshi-spectral_seq}{article}{
  author={Ohsawa, Takeo},
  author={Takegoshi, Kensh\={o}},
  title={Hodge spectral sequence on pseudoconvex domains},
  journal={Math. Z.},
  volume={197},
  date={1988},
  number={1},
  pages={1--12},
  issn={0025-5874},
  review={\MR {917846}},
  doi={10.1007/BF01161626},
}

\bib{Siu}{article}{
  author={Siu, Yum Tong},
  title={Complex-analyticity of harmonic maps, vanishing and Lefschetz theorems},
  journal={J. Differential Geom.},
  volume={17},
  date={1982},
  number={1},
  pages={55--138},
  issn={0022-040X},
  review={\MR {658472}},
}

\bib{Takegoshi_higher-direct-images}{article}{
  author={Takegoshi, Kensh\={o}},
  title={Higher direct images of canonical sheaves tensorized with semi-positive vector bundles by proper K\"{a}hler morphisms},
  journal={Math. Ann.},
  volume={303},
  date={1995},
  number={3},
  pages={389--416},
  issn={0025-5831},
  review={\MR {1354997}},
  doi={10.1007/BF01460997},
}


    }
  \end{biblist}
\end{bibdiv}
\end{document}